\documentclass{mcom-l}

\usepackage{graphicx, caption, subcaption}
\usepackage[colorlinks,linkcolor=blue,citecolor=blue]{hyperref}
\usepackage[hyperpageref]{backref}
\usepackage{times,amsmath,amsbsy,amssymb,amscd,mathrsfs, bm}
\usepackage{makecell}
\usepackage{booktabs}
\usepackage{multicol,multirow}

\newtheorem{theorem}{Theorem}[section]
\newtheorem{lemma}[theorem]{Lemma}

\theoremstyle{definition}

\newtheorem{corollary}[theorem]{Corollary}

\theoremstyle{remark}
\newtheorem{remark}[theorem]{Remark}

\numberwithin{equation}{section}

\newcommand{\dd}{\,{\rm d}}
\newcommand{\dual}[1]{\left\langle {#1} \right\rangle}
\newcommand{\mX}{\mathcal X}
\newcommand{\IX}{\mathcal{I}_{\mathcal X}}

\newcommand{\IV}{\mathcal{I}_{\mathcal V}}
\newcommand{\IQ}{\mathcal{I}_{\mathcal Q}}

\DeclareMathOperator{\sym}{sym}

\begin{document}

\title[Accelerated Gradient and Skew-Symmetric Splitting Methods]{Accelerated Gradient and Skew-Symmetric Splitting Methods for a Class of Monotone Operator Equations}

\author{Long Chen}
\address{Department of Mathematics, University of California at Irvine, Irvine, CA 92697, USA}
\email{chenlong@math.uci.edu}
\thanks{L. Chen and J. Wei are supported by National Science Foundation DMS-2012465 and DMS-2309785.}

\author{Jingrong Wei}
\address{Department of Mathematics, University of California at Irvine, Irvine, CA 92697, USA}
\email{jingronw@uci.edu}

\date{DATE}



\begin{abstract}
	A class of monotone operator equations, which can be decomposed into sum of the gradient of a strongly convex function and a linear and skew-symmetric operator, is considered in this work. Based on discretization of the generalized gradient flow, gradient and skew-symmetric splitting (GSS) methods are proposed and proved to converge in linear rates. To further accelerate the convergence, an accelerated gradient flow is proposed and accelerated gradient and skew-symmetric splitting (AGSS) methods are developed, which extends the acceleration among the existing works on the convex minimization to a more general class of monotone operator equations. In particular, when applied to smooth saddle point systems with bilinear coupling, a linear convergent method with optimal lower iteration complexity is proposed. The robustness and efficiency of GSS and AGSS methods are verified via extensive numerical experiments.
\end{abstract}

\maketitle


\section{Introduction}
In this paper, we consider iterative methods for solving a class of monotone operator equations
\begin{equation}\label{eq:Ax}
\mathcal A(x) = 0, \quad x\in \mathbb R^n,
\end{equation}
where  $\mathcal A$ is strongly monotone: there exists $\mu > 0$,  for any $x, y \in \mathbb{R}^n,$
\begin{equation}\label{eq: monotone}
(\mathcal{A}(x) - \mathcal{A}(y), x-y) \geq \mu \|x-y\|^2.
\end{equation}
We further assume $\mathcal A(x)$ can be decomposed into
\begin{equation}\label{eq:dec}
\mathcal A(x) = \nabla F(x) + \mathcal N x,
\end{equation}
with a strongly convex $C^1$ function $F$, and a linear and skew-symmetric operator $\mathcal N$, i.e., $\mathcal N^{\intercal} = -\mathcal N$. Then $\mathcal A$ is Lipschitz continuous with constant $L_{\mathcal A}\leq L_F + \| \mathcal N \|$, where $L_F$ is the Lipschitz constant of $\nabla F$. Therefore $\mathcal A$ is monotone and Lipschitz continuous which is also known as inverse-strongly monotonicity~\cite{browder1967construction,liu1998regularization} or co-coercitivity~\cite{zhu1996co}. 
Consequently~\eqref{eq:Ax} has a unique solution $x^{\star}$~\cite{rockafellar1976monotone}. 

\subsection{Contribution}
The goal of this paper is to design iterative methods for computing the solution $x^{\star}$ of \eqref{eq:Ax} based on the splitting~\eqref{eq:dec}, which will be called Gradient and skew-Symmetric Splitting (GSS) methods and its accelerated version Accelerated Gradient and skew-Symmetric Splitting (AGSS) method. Our contribution is summarized as below.  

\begin{itemize}

    \item   Acceleration techniques in optimization have traditionally been limited to convex optimization problems~\cite{nesterov1983method,polyak1964some}. This paper significantly broadens the scope of acceleration beyond its conventional boundaries by extending to a more general class of monotone operator equations. 
    
    \item We leverage ordinary differential equation (ODE) models as a foundation for crafting discrete optimization schemes. Our framework encompasses various numerical methods tailored for solving \eqref{eq:Ax}. Notably, in comparison to the traditional proximal point methods, our novel GSS method and AGSS method offer explicit and practical schemes.

    
    \item For solving non-Hermitian positive definite linear systems $\mathcal Ax  = b$, our new AGSS method achieves the same accelerated rate as Hermitian/skew-Hermitian splitting method (HSS)~\cite{bai2003hermitian} without treating the symmetric part implicitly, and more importantly can handle non-linear problems.

    \item When applied to smooth strongly-convex-strongly-concave saddle point systems with bilinear coupling, our new AGSS methods are few first-order algorithms achieving optimal lower iteration complexity. 
\end{itemize}

The robustness and efficiency of GSS and AGSS methods are verified via extensive numerical experiments: quadratic problems, convection-diffusion partial differential equations, and empirical risk minimization.

\subsection{Gradient and Skew-symmetric Splitting Methods}
We shall develop our iterative methods based on discretization of ODE flows. Namely we treat $x(t)$ as a continuous function of an artificial time variable $t$ and design ODE systems $x'(t) = \mathcal G(x(t))$ so that the $x^{\star}$ is a stable equilibrium point of the corresponding dynamic system, i.e., $\mathcal G(x^*)=0$ and $\lim_{t\to \infty} x(t) = x^*$. Then we apply ODE solvers to obtain various iterative methods. By doing this way, we can borrow the stability analysis of dynamic systems and convergence theory of ODE solvers. 

The most popular choice is the generalized gradient flow:
\begin{equation}\label{eq:intro gradient flow}
x^{\prime}(t) = - \mathcal{A}(x(t)) = - \nabla F( x(t)) - \mathcal N(x(t)).
\end{equation}
One can easily show the exponential stability of $x^{\star}$ using the monotonicity of $\mathcal A$. Although $\mathcal A$ is not the gradient of some scalar function when $\mathcal N\neq 0$, we still use the name convention for the convex case $\mathcal N = 0$. 
%

\begin{table}[htp]
\small
	\centering
	\caption{GSS methods for solving $x^{\prime}(t) = - \nabla F( x(t)) - \mathcal N(x(t))$.} 
	\renewcommand{\arraystretch}{1.25}
	\begin{tabular}{@{} c c c  c c c c @{}}
	\toprule
  $\nabla F$   & $\mathcal N$  & Algorithms & Complexity  &  Analysis
\smallskip \\  \hline 

 Implicit    & Implicit & \eqref{intro:implicit}, \eqref{eq:implicit} & $\ln ( |\ln \epsilon |) $  & \eqref{eq:IEproof} &

\smallskip \\

  Explicit    & Explicit &  \eqref{intro:explicit}, \eqref{eq: EE gradient} & $\kappa^2 (\mathcal A) |\ln \epsilon |$ & \eqref{eq:EEproof 1} and \eqref{eq:EEproof 2}

\smallskip \\

  Explicit    & Implicit &  \eqref{eq: semi-implicit Euler} & $\kappa (F)  |\ln \epsilon |$ & Theorem \ref{thm:IMEX}

\smallskip \\

 \makecell{  Implicit \\
        ( $\nabla F = \mu I $ ) } & AOR & \eqref{eq: GS method y},  \eqref{eq: GS method 2}   & $\tilde{\kappa}( B^{\sym})|\ln \epsilon |$  & Theorem \ref{thm: linear convergence for the shifted skew-symmetric linear system}  

\smallskip \\
 
 Explicit    & AOR & \eqref{eq: intro nonlinear SOR 1}, \eqref{eq: nonlinear SOR 1}, \eqref{eq: nonlinear SOR 2} & $(\tilde{\kappa}( B^{\sym}) + \kappa (F)) |\ln \epsilon |$ & Theorem \ref{thm: linear convergence for generalized nonlinear system} 

\smallskip \\
\bottomrule
	\end{tabular}
	\label{tab: GSS methods}
\end{table}

Discretization of the generalized gradient flow~\eqref{eq:intro gradient flow} leads to iterative methods for solving~\eqref{eq:Ax}. Various methods will be obtained by treating implicit or explicit discretization of $\nabla F$ and $\mathcal N$, and are summarized in Table \ref{tab: GSS methods}. For easy of references, we list equation numbers in both introduction and later sections. 

We explain accelerated overrelaxation~\cite{hadjidimos1978accelerated,hadjidimos2000successive} method (AOR) listed in Table \ref{tab: GSS methods}. Let $B^{\intercal} = {\rm upper}(\mathcal N)$ be the upper triangular part of $\mathcal N$ and $B^{\sym} = B + B^{\intercal}$ be a symmetrization of $B$. Then based on the splitting
 $\mathcal N = B^{\sym}-2B$ or $\mathcal N = 2 B^{\intercal} - B^{\sym}$, the symmetric part will be treat explicitly and $B$ or $B^{\intercal}$ implicitly. 

The computation complexity is measured by the iteration step to reach an $\epsilon$-solution $x$ such that $\|x - x^{\star}\|^2< \epsilon$, up to a constant factor. Condition number $\kappa(\mathcal A) = L_{\mathcal A}/\mu$, $\kappa(F) = L_F/\mu$, and $\tilde{\kappa}( B^{\sym}) = \| B^{\sym}\|/\mu$.

We briefly go through methods in Table \ref{tab: GSS methods}. The implicit Euler scheme with step size $\alpha_k$ is equivalent to the  proximal method 
\begin{equation}\label{intro:implicit}
(I + \alpha_k \mathcal A) (x_{k+1}) =  x_{k},
\end{equation}
which is unconditionally stable for any $\alpha_k > 0$. For increasing step size $\alpha_k$, the implicit Euler scheme can even achieve super-linear convergence. The explicit Euler scheme gives the generalized gradient descent method:
\begin{equation}\label{intro:explicit}
    x_{k+1} = x_k - \alpha_k \mathcal{A}(x_k).
\end{equation}
For $\displaystyle \alpha_k =  \mu/L_{\mathcal A}^2$,  \eqref{intro:explicit} achieves the linear rate
$$
\| x_{k+1} - x^{\star} \|^2 \leq
 \left( 1- 1/\kappa^2(\mathcal A)\right )\| x_{k} - x^{\star} \|^2.
$$
But this linear rate is pretty slow when $\kappa(\mathcal A)\gg 1$ since the dependence is $\kappa^2(\mathcal A)$. 
In contrast, for convex optimization, i.e., $\mathcal N = 0$, the rate for  the gradient descent method is  $1-c/\kappa(F)$ for step size $\alpha_k = 1/L_F$ and $1-c/\sqrt{\kappa(F)}$ for accelerated gradient methods~\cite{nesterov2003introductory}. 

Based on the splitting $\mathcal N = B^{\sym}-2B$ or $\mathcal N = 2 B^{\intercal} - B^{\sym}$, we propose a new method called Gradient and skew-Symmetric Splitting (GSS) method: 
\begin{equation}\label{eq: intro nonlinear SOR 1}
    \frac{x_{k+1} - x_k}{\alpha} = -\left (\nabla F(x_k)+ B^{\sym}x_k - 2B x_{k+1}\right),
\end{equation}
Notice that  $B$ is lower triangular, $x_{k+1}$ in ~\eqref{eq: intro nonlinear SOR 1} can be computed via forward substitution. We prove that GSS method \eqref{eq: intro nonlinear SOR 1} achieves the linear convergence rate, for $ \alpha = \displaystyle \min \left\{\frac{1}{4\|B^{\sym}\|}, \frac{1}{4L_F} \right\}$,
\begin{equation*}
\| x_k - x^{\star}\|^2 \leq \left (\frac{1}{1 + 4/\max \{\tilde{\kappa}(B^{\sym}), \kappa(F)\} }\right )^{k} 6\| x_0 - x^{\star}\|^2.
\end{equation*}

In summary, in GSS, AOR is used to accelerate the skew-symmetric part and improve the dependence of condition number from $\kappa^2$ to $\kappa$. The convergence rate is validated through a numerical example on quadratic problems in Section \ref{sec: Quadratic problems}.

\subsection{Accelerated Gradient and Skew-symmetric Splitting Methods}
To further accelerate the convergence rate, we introduce an accelerated gradient flow 
\begin{equation}\label{eq:intro AG}
\left \{\begin{aligned}
     x^{\prime} &= y - x ,\\
     y^{\prime} & = x - y - \mu^{-1}(\nabla F(x) + \mathcal N y).
\end{aligned}\right .
\end{equation}
Comparing with the accelerated gradient flow in~\cite{luoDifferentialEquationSolvers2021a} for convex optimization, the difference is the gradient and skew-symmetric splitting $\mathcal A(x) \rightarrow \nabla F(x) + \mathcal N y$. 

\begin{table}[htp]
\small
	\centering
	\caption{AGSS methods for solving $x^{\prime}= y - x , \quad
     y^{\prime} = x - y - \mu^{-1}(\nabla F(x) + \mathcal N y)$.} 
	\renewcommand{\arraystretch}{1.25}
	\begin{tabular}{@{} c c c  c c c c @{}}
	\toprule
  $\nabla F$   & $\mathcal N$  & Algorithm & Complexity  &  Analysis
\smallskip \\  \hline 
  
  Explicit    & Implicit & \eqref{eq: intro IMEX}, \eqref{eq:semi-implicit} & $\sqrt{\kappa(F)} |\ln \epsilon |$  & Theorem \ref{thm: linear convergence of AGSS IMEX scheme}  

\smallskip \\

  Explicit    & Implicit (inexact) & \eqref{eq:inexact-implicit} & $\sqrt{\kappa(F)} |\ln \epsilon |$  & Corollary \ref{cor:inexact IMEX}  

\smallskip \\

   Explicit    & AOR & \eqref{eq:intro-agss}, \eqref{eq:explicit} & $\left (\tilde{\kappa}( B^{\sym}) + \sqrt{\kappa (F)} \right )|\ln \epsilon |$ & Theorem \ref{thm: linear convergence of AGSS explicit scheme}

\smallskip \\

\bottomrule
	\end{tabular}
	\label{tab: AGSS methods}
\end{table}

We propose several iterative schemes based on discretization of \eqref{eq:intro AG} and name them Accelerated Gradient and skew-Symmetric Splitting (AGSS) methods. Provided that a fast solver for computing $ (\beta  I +  \mathcal N)^{-1}$ is available, we give an IMplicit-EXplicit (IMEX) scheme for~\eqref{eq:intro AG}:
\begin{equation}\label{eq: intro IMEX}
\begin{aligned}
   \frac{\hat{x}_{k+1}-x_k}{\alpha_k} &=  y_k - \hat{x}_{k+1},  \\
  \frac{y_{k+1}-y_k}{\alpha_k} &=  \hat{x}_{k+1} - y_{k+1} - \mu^{-1} \left( \nabla F(\hat{x}_{k+1}) + \mathcal Ny_{k+1}\right) , \\
  \frac{x_{k+1}-x_k}{\alpha_k} &= y_{k+1} - x_{k+1}.
\end{aligned}
\end{equation}
The scheme \eqref{eq: intro IMEX} is implicit for $y_{k+1}$ as each iteration needs to solve a linear equation $(\beta  I + \mathcal N)y_{k+1} = b(\hat x_{k+1}, y_k)$ associated to a shifted skew-symmetric system with $\beta = 1 +\mu/\alpha_k$. For fixed step size $\alpha_k = 1/\sqrt{\kappa(F)}$, the convergence rate is accelerated to
\begin{equation*}
\|x_{k+1}-x^{\star}\|^2 + \|y_{k+1}-x^{\star}\|^2 \leq \left (\frac{1}{1+1/\sqrt{\kappa(F)}}\right )^{k} 2 \mathcal{E}_0/\mu,
\end{equation*}
where $\mathcal{E}_0  = D_F(x_0, x^{\star}) + \frac{\mu}{2}\|y_0-x^{\star}\|^2$ and $D_F$ is the Bregman divergence of $F$. For linear systems, comparing with Hermitian/skew-Hermitian splitting method (HSS)~\cite{bai2003hermitian}, we achieve the same accelerated rate without treating the symmetric part implicitly and consequently improve the efficiency. More importantly, AGSS can handle non-linear problems while HSS is restricted to linear algebraic systems. 

We also allow an inexact solver for approximating $( \beta I +  \mathcal N)^{-1}$ using a perturbation argument to control the accumulation of the inner solve error. Formal results are presented in Section  \ref{sec: Inexact Solver for the Shifted Skew Symmetric System}. In practice, we showcase the efficiency of the IMEX scheme through computations on a convection-diffusion model in Section \ref{sec: Convection-diffusion model}.
%

To fully avoid computing $( \beta I +   \mathcal N)^{-1}$, we provide an explicit AGSS scheme combining acceleration of $F$ and AOR technique for $\mathcal N$: 
\begin{equation}\label{eq:intro-agss}
\begin{aligned}
  \frac{\hat x_{k+1}-x_k}{\alpha} &=  y_k - \hat x_{k+1}, \\
   \frac{y_{k+1}-y_k}{\alpha} &= \hat x_{k+1} - y_{k+1}- \mu^{-1}\left( \nabla F(\hat x_{k+1}) + B^{\sym} y_k - 2B y_{k+1} \right), \\
   \frac{x_{k+1}-x_{k}}{\alpha} &= y_{k+1} -\frac{1}{2} (x_{k+1} + \hat x_{k+1}).
\end{aligned}
\end{equation}
For the step size $\displaystyle \alpha= \min \left \{\frac{\mu}{2L_{B^{\sym}}}, \sqrt{\frac{ \mu}{2L}}\right \}$, we obtain the accelerated linear convergence for scheme \eqref{eq:intro-agss}
$$
\|x_{k+1}-x^{\star}\|^2 \leq \left (\frac{1}{1 + 1/\max\{4\tilde{\kappa}(B^{\sym}), \sqrt{8\kappa(F)}\}}\right )^{k}\frac{2\mathcal E_0^{\alpha B}}{\mu}
$$
where $\mathcal E^{\alpha B}_0  = D_F(x_0, x^{\star}) + \frac{1}{2}\|y-x^{\star}\|^2_{\mu I - 2\alpha B^{\sym}}$ is nonnegative according to our choice of the step size.

In summary, AGSS based on the accelerated gradient flow \eqref{eq:intro AG} can be applied to the convex part to further improve the dependence of $\kappa(F)$ to $\sqrt{\kappa(F)}$. 

\subsection{Nonlinear Saddle Point Systems with bilinear coupling}

The proposed accelerated methods has wide applications. As an example, we apply to strongly-convex-strongly-concave saddle point systems with bilinear coupling:
    \begin{equation}\label{eq:saddle}
        \min_{u\in \mathbb R^m} \max_{p\in \mathbb R^n} \mathcal L(u,p) := f(u) - g(p) + (Bu, p),
    \end{equation}
where $f, g$ are smooth, strongly convex functions with Liptschitz continuous gradients and $B\in \mathbb R^{n\times m}, m\geq n,$ is a full rank matrix. By a sign change, the saddle point system $\nabla \mathcal L(u,p)=0$ becomes~\eqref{eq:Ax} with $x = (u,p)^{\intercal}$, $F(u,p) = f(u) + g(p)$ and $\mathcal N = \begin{pmatrix} 0 & \;  B^{\intercal} \\ -B & \;  0 \end{pmatrix}$. 

Due to the $2\times 2$ block structure, preconditioners can be used to further improve the convergence. Let $\IV, \IQ$ be metric induced by designated symmetric and positive definite matrices and $\IX = {\rm diag} (\IV, \IQ)$. The convexity constant of $f, g$  is denoted by $\mu_f, \mu_g$ under the metric $\IV, \IQ$, respectively. Denote by $L_S = \lambda_{\max}(\IQ^{-1}B\IV^{-1}B^{\top})$ and $\tilde{\kappa}_{\IX}(\mathcal N)=\sqrt{L_S/(\mu_f\mu_g)}$. We summarize GSS and AGSS methods for solving the saddle point problem \eqref{eq:saddle} in Table \ref{tab: saddle point algorithms}. 

\begin{table}[htp]
\footnotesize
	\centering
	\caption{Summary of GSS and AGSS for solving $\min_u \max_p f(u) - g(p) + (Bu, p)$.} 
	\renewcommand{\arraystretch}{1.25}
	\begin{tabular}{@{} c c c c  c c c c @{}}
	\toprule
Method &   $\nabla f, \nabla g$   & $\mathcal N$  & Algorithm & Complexity  &  Analysis
 \smallskip \\  \hline

 GSS   & Explicit  & AOR & \eqref{eq:GSS_saddle} & $\kappa_{\IV}(f)+ \tilde{\kappa}_{\IX}(\mathcal N)+\kappa_{\IQ}(g)|\ln \epsilon |$ & Theorem \ref{thm: linear convergence for GSS_saddle}
 
\smallskip \\

GSS   & Implicit  & AOR & \eqref{eq:intro IMEX saddle prox}, \eqref{eq:IMEX saddle prox}  & $\tilde{\kappa}_{\IX}(\mathcal N) |\ln \epsilon |$ & Theorem \ref{thm: linear convergence for the saddle prox}
 
\smallskip \\

AGSS &    Explicit    & AOR & \eqref{eq:intro explicit saddle}, \eqref{eq:explicit saddle} & $\sqrt{\kappa_{\IV}(f)+ \tilde{\kappa}_{\IX}^2(\mathcal N)+\kappa_{\IQ}(g)} |\ln \epsilon |$  &  Theorem \ref{thm: linear convergence of explicit scheme saddle}

\smallskip \\

AGSS    & Explicit    & Implicit & \eqref{eq: intro semi-implicit saddle}, \eqref{eq:semi-implicit saddle} & $\sqrt{\kappa_{\IV}(f)+\kappa_{\IQ}(g)} |\ln \epsilon | $ & Theorem \ref{thm: linear convergence of saddle IMEX scheme}

\medskip \\

\bottomrule
	\end{tabular}
	\label{tab: saddle point algorithms}
\end{table}

Let $x = (u,p)$ and $y = (v,q)$. Explicit AGSS method~\eqref{eq:intro-agss} applied to the saddle point system \eqref{eq:saddle} but with preconditioners $\IV^{-1}$ and $\IQ^{-1}$ becomes:
\begin{equation}\label{eq:intro explicit saddle}
\begin{aligned}
  \frac{\hat x_{k+1}-x_k}{\alpha} &= y_k - \hat x_{k+1}, \\
  \frac{v_{k+1}-v_k}{\alpha} &=\hat{u}_{k+1} - v_{k+1} -\mu_f^{-1} \IV^{-1}\left( \nabla f(\hat{u}_{k+1}) + B^{\intercal}q_{k}\right) , \\
     \frac{q_{k+1}-q_k}{\alpha} &=  \hat{p}_{k+1} - q_{k+1}-\mu_g^{-1} \IQ^{-1} \left( \nabla g(\hat{p}_{k+1}) -2Bv_{k+1} + Bv_k \right) , \\
   \frac{x_{k+1}-x_{k}}{\alpha} &=y_{k+1} - x_{k+1} + \frac{1}{2}(x_{k+1} - \hat x_{k+1}).
\end{aligned}
\end{equation}
For the step size $\alpha =  1/\max \left \{2\tilde{\kappa}_{\IX}(\mathcal N), \sqrt{2\kappa_{\IV}(f)}, \sqrt{2\kappa_{\IQ}(g)}\right \}$, we achieve the accelerated rate
\begin{equation*}
\begin{aligned}
&\mu_f \|u_{k}-u^{\star}\|_{\IV}^2 + \mu_g\|p_{k}-p^{\star}\|_{\IQ}^2  \\
&\leq \left (1+ 1/\max \left \{4\tilde{\kappa}_{\IX}(\mathcal N), 2\sqrt{2\kappa_{\IV}(f)}, 2\sqrt{2\kappa_{\IQ}(g)}\right \}\right)^{-k} 2\mathcal{E}_0^{\alpha B}.
\end{aligned}
\end{equation*}
where $\mathcal{E}^{\alpha B}(x_0, y_0) > 0$ is the initial value of a tailored Lyapunov function defined as \eqref{eq: acc discrete Lyapunov saddle}. In particular when $\IV$ and $\IQ$ is set to identity matrices, $\tilde{\kappa}(\mathcal N) = \|B\|/\sqrt{\mu_f\mu_g}$. The linear rate matches the lower complexity bound of the first-order algorithms for saddle point problems established in~\cite{zhang2022lower}. We showcase the optimality of our algorithm and compare it with others by solving the empirical risk minimization in Section \ref{sec: Empirical risk minimization}.

If the skew-symmetric part is treated implicitly, we obtain an IMEX AGSS scheme:
\begin{equation}\label{eq: intro semi-implicit saddle}
\begin{aligned}
   \frac{\hat{x}_{k+1}-x_k}{\alpha_k} &= y_k - \hat{x}_{k+1},  \\
       \frac{v_{k+1}-v_k}{\alpha_k} &= \hat{u}_{k+1} - v_{k+1}-  \mu_f^{-1} \IV^{-1}\left( \nabla f(\hat{u}_{k+1}) + B^{\intercal}q_{k+1}\right) , \\
    \frac{q_{k+1}-q_k}{\alpha_k} &= \hat{p}_{k+1} - q_{k+1} - \mu_g^{-1} \IQ^{-1} \left( \nabla g(\hat{p}_{k+1}) -Bv_{k+1}\right) , \\
   \frac{x_{k+1}-x_k}{\alpha_k} &= y_{k+1} - x_{k+1}.
\end{aligned}
\end{equation}
which requires solving a linear saddle point system
\begin{equation}\label{eq: intro I+N saddle}
\begin{pmatrix}
 (1+\alpha_k)\mu_f \IV & \alpha_k B^{\intercal} \\
 -\alpha_k B & (1+\alpha_k) \mu_g \IQ
\end{pmatrix}\begin{pmatrix}
 v_{k+1} \\
 q_{k+1}
\end{pmatrix}  = b(\hat{x}_{k+1},y_k),
\end{equation}
with $b(\hat{x}_{k+1},y_k) = \IX( y_k + \alpha_k \hat x_{k+1} )- \alpha_k \nabla F(\hat x_{k+1})$. 
Consequently the linear rate, $$\left (1+1/\max \left \{\sqrt{\kappa_{\IV}(f)}, \sqrt{\kappa_{\IQ}(g)} \right \} \right)^{-1}$$ is free of $\tilde{\kappa}_{\IX}(\mathcal N)$. Inexact inner solvers for solving \eqref{eq: intro I+N saddle} is allowed.

Given proximal operators of $f$ and $g$, we can treat $F$ implicitly and obtain the primal-dual method with AOR:
\begin{equation}\label{eq:intro IMEX saddle prox}
\begin{aligned}
    u_{k+1} &= \operatorname{prox}_{\alpha/\mu_f f, \IV}(u_k - \alpha \mu_f^{-1}\IV^{-1}B^{\intercal}p_k), \\
   p_{k+1}  &= \operatorname{prox}_{\alpha/\mu_g g, \IQ}(p_k -  \alpha \mu_g^{-1} \IQ^{-1}B(u_k - 2u_{k+1})),
\end{aligned}
\end{equation}
where the generalized proximal operator is defined as
\begin{equation*}
\begin{aligned}
    \operatorname{prox}_{\gamma f, \IV}(v)&:=\underset{u}{\operatorname{argmin}} ~f(u)+\frac{1}{2 \gamma}\|u-v\|^2_{ \IV}, \\
     \operatorname{prox}_{\sigma g, \IQ}(q)&:=\underset{p}{\operatorname{argmin}} ~g(p)+\frac{1}{2 \sigma}\|p-q\|^2_{\IQ}.
\end{aligned}
\end{equation*}
For  $ \alpha = 1/(2\tilde{\kappa}_{\IX}(\mathcal N))$, we have
\begin{equation*}
\| x_k - x^{\star}\|^2_{ \mu \IX} \leq \left (1 + 1/(2\tilde{\kappa}_{\IX}(\mathcal N)\right )^{-k} 3\| x_0 - x^{\star}\|^2_{  \mu \IX}.
\end{equation*}
Again when $\IV$ and $\IQ$ are identity matrices, $\tilde{\kappa}(\mathcal N) = \|B\|/\sqrt{\mu_f\mu_g}$ is indeed the optimal lower bound shown in~\cite{zhang2022lower}. 

%

\subsection{Related Work}
More specific class of monotone operator equation is particularly of interest, since it is closely related to convex optimization and saddle point problems. For instance, when $\mathcal N = 0$,~\eqref{eq:Ax} becomes the convex minimization problem 
\begin{equation}\label{eq: convex min}
    \min_x F(x).
\end{equation}
We refer to the book~\cite{nesterov2003introductory} on the first-order methods for solving~\eqref{eq: convex min} especially the accelerated gradient method which achieved the linear rate $1-1/\sqrt{\kappa(F)}$ with optimal lower complexity bound. Most acceleration methods are restricted to convex optimization and extension to general monotone operators is rare. We significantly broaden the scope of acceleration. 

Another important application is the strongly-convex-strongly-concave saddle point problem with bilinear coupling \eqref{eq:saddle}. 
%
In~\cite{zhang2022lower}, it is shown that if the duality gap $\Delta(u,p) := \max_q \mathcal L(u, q) -  \min_v \mathcal L(v, p) $ 
is required to be bounded by $\epsilon$, the number of iterations for the  first-order methods is at least 
\begin{equation}\label{eq:lowercomplexity}
\Omega\left(\sqrt{\kappa(f)+\tilde \kappa^2(\mathcal N)+\kappa(g)}  |\ln \epsilon |\right)
\end{equation}
where $\kappa(f), \kappa(g)$ are condition numbers of $f, g$ respectively, $\tilde{\kappa}(\mathcal N) = \|B\|/\sqrt{\mu_f \mu_g}$ measuring the coupling, and $\Omega$ means the iteration complexity bounded below asymptotically, up to a constant. Only few optimal first-order algorithms have been developed recently~\cite{metelev2024decentralized,thekumparampil2022lifted}. If further the proximal oracles for $f, g$ are allowed, the lower complexity bound of the class of first-order algorithms is $\displaystyle \Omega\left (\tilde{\kappa}(\mathcal N)|\ln \epsilon| \right)$ and  proximal methods achieving this lower bound can be found in~\cite{chambolle2011first,chambolle2016ergodic}. Our new method AGSS applied to the saddle point problems becomes first-order algorithms achieving optimal lower iteration complexity and achieve robust performance; see Section \ref{sec: Empirical risk minimization} for the comparison with other optimal complexity methods~\cite{kovalev2022accelerated,li2023nesterov,mokhtari2020convergence,thekumparampil2022lifted}. 

When $\mu_g = 0$, transformed primal-dual (TPD) methods ~\cite{chen2023transformed,chen2024transformed} can be combined with the accelerated gradient flow. Here, no transformation is included in the AGSS algorithms. 




One more important class is for linear operator $\mathcal A$ which has the following decomposition:
$$
\mathcal A = \mathcal A^{\rm s} + \mathcal N,
$$
where $\mathcal A^{\rm s} = (\mathcal A + \mathcal A^{\intercal})/2$ is the symmetric (Hermitian for complex matrices) part and $\mathcal N = (\mathcal A - \mathcal A^{\intercal})/2$ is the skew-symmetric part. The condition $\mathcal A$ is monotone is equivalent to  $\mathcal A^{\rm s}$ is symmetric and positive definite (SPD) and $\lambda_{\min}( \mathcal A^{\rm s})\geq \mu$. Bai, Golub, and Ng in~\cite{bai2003hermitian} proposed the Hermitian/skew-Hermitian splitting method (HSS) for solving general non-Hermitian positive definite linear systems $\mathcal Ax  = b$:
\begin{equation}\label{eq:HSS}
    \begin{aligned}
    (\alpha I+ \mathcal A^{\rm s}) x_{ k+\frac{1}{2}} &=(\alpha I-\mathcal N) x_{k}+b \\ 
    (\alpha I+\mathcal N) x_{k+1} &=(\alpha I-\mathcal A^{\rm s}) x_{k+\frac{1}{2}}+b.
\end{aligned}
\end{equation}
The iterative method~\eqref{eq:HSS} solves the equations for the symmetric (Hermitian) part and skew-symmetric (skew-Hermitian) part alternatively. For the HSS method~\eqref{eq:HSS}, efficient solvers for linear operators $(\alpha I + \mathcal A^{\rm s})^{-1}$ and $(\alpha I+ \mathcal N)^{-1}$ are needed. A linear convergence rate of $\displaystyle \frac{\sqrt{\kappa (\mathcal A^{\rm s})}-1}{\sqrt{\kappa (\mathcal A^{\rm s})}+1}$ can be achieved for an optimal choice of parameter $\alpha$. Several variants of HSS are derived and analyzed in~\cite{bai2007accelerated,bai2007successive}. 

For linear systems, comparing with HSS, AGSS achieves the same accelerated rate without treating the symmetric part implicitly, i.e., no need to compute $(\alpha I + \mathcal A^{\rm s})^{-1}$, and consequently significantly improve the efficiency. More importantly, AGSS can handle non-linear problems while HSS is restricted to linear algebraic systems only.

For the rest of the paper, we first review convex function theory and Lyapunov analysis theory in Section \ref{sec: preliminary}, as basic preliminaries for our proofs. The generalized gradient flow and convergence of GSS methods for the considered monotone equation are proposed in Section \ref{sec: gradient methods}. Then in Section \ref{sec:AGSS flow and schemes}, we extend the preconditioned accelerated gradient flow in convex minimization and derive AGSS methods with accelerated linear rates. As applications, we give optimal algorithms for strongly-convex-strongly-concave saddle point system with bilinear coupling in Section \ref{sec:saddle point systems}. In Section \ref{sec: numerical ex}, we validate the algorithms with numerical examples. Concluding remarks are addressed in Section \ref{sec: conclusion}.

\section{Preliminary}\label{sec: preliminary}
In this section, we follow~\cite{nesterov2003introductory} to introduce notation and  preliminaries in convex function theory. We also briefly review a unified convergence analysis of first order convex optimization methods via strong Lyapunov functions established in~\cite{chen2021unified}. 

\subsection{Convex Functions}
Let $\mX$ be a Hilbert space with inner product $(\cdot, \cdot)$ and norm $\|\cdot \|$. $\mX^{\prime}$ is the linear space of all linear and continuous mappings, which is called the dual space of $\mX$, and $\langle \cdot, \cdot \rangle$ denotes the duality pair between $\mX^{\prime}$ and $\mX$. For any proper closed convex and $C^1$ function $F: \mX \rightarrow \mathbb{R}$, the Bregman divergence of $F$ is defined as
$$D_{F}(x,y) : = F(x)-F(y)-\langle \nabla F(y), x-y\rangle.$$
We say $F \in \mathcal{S}_{\mu}$ or $\mu$-strongly convex with $\mu > 0$ if $F$ is differentiable and  
$$
D_{F}(x, y)  \geqslant \frac{\mu}{2}\left\|x-y\right\|^{2}, \quad \forall x, y \in \mX.
$$
 In addition, denote by $F \in \mathcal{S}_{\mu, L}$ if $F \in \mathcal{S}_{\mu}$ and there exists $L>0$ such that
$$
D_{F}(x, y) \leqslant \frac{L}{2}\left\|x-y\right\|^{2}, \quad \forall x, y \in \mX.
$$
For fixed $y \in \mX, D_{F}(\cdot, y)$ is convex as $F$ is convex and 
$$
\nabla D_{F}(\cdot, y) = \nabla F(\cdot) - \nabla F(y).
$$
In general $D_F(x, y)$ is non-symmetric in terms of $x$ and $y$. 
A symmetrized Bregman divergence is defined as
\begin{equation*}\label{eq: symmetrized Bregman divergenve}
    \langle\nabla F(y)-\nabla F(x), y-x\rangle=D_{F}(y, x)+D_{F}(x, y).
\end{equation*}
By direct calculation, we have the following three-terms identity.
\begin{lemma}[Bregman divergence identity~\cite{chen1993convergence}]
If function $F: \mX \rightarrow \mathbb{R}$ is differentiable, then for any $x, y, z \in \mX$, it holds that
\begin{equation}\label{eq: Bregman divergence identity}
   \langle\nabla F(x)-\nabla F(y), y-z\rangle=D_{F}(z, x)-D_{F}(z, y)-D_{F}(y, x). 
\end{equation}
\end{lemma}
When $F(x) = \frac{1}{2}\|x\|^2$, identity~\eqref{eq: Bregman divergence identity} becomes 
\begin{equation*}
 (x-y, y-z) =  \frac{1}{2}\|z-x\|^2 - \frac{1}{2}\|z-y\|^2 - \frac{1}{2}\|y-x\|^2.
\end{equation*}

For a non-smooth function $F$, the subgradient is a set-valued function defined as
$$\partial F(x):=\left\{\xi \in \mathcal X^{\prime}: F(y)-F(x) \geqslant\langle \xi, y-x\rangle, \quad \forall y \in \mathcal X \right\}.$$
For a proper, closed and convex function $F: \mathcal X \to \mathbb{R}$, $\partial F(x)$ is a nonempty bounded set for $x\in \mathcal X$~\cite{nesterov2003introductory}. It is evident that the subgradient for smooth $F$ is a single-valued function reduced to the gradient, that is $\partial F(x) = \{\nabla F(x)\}$. We extend the set $\mathcal S_{\mu}$ using the notion of subgradient: $F \in \mathcal S_{\mu}$ if for all $x, y \in \mathcal X$,
$$F(y) \geqslant F(x)+\langle \xi, y-x\rangle+\frac{\mu}{2}\|x-y\|^2, \quad \forall \xi \in \partial F(x).$$

Given a convex function $F$, define the proximal operator of $F$ as
\begin{equation}\label{eq:proxl2}
     \operatorname{prox}_{\gamma f, \IV}(y):=\underset{x}{\operatorname{argmin}} ~F(x)+\frac{1}{2 \gamma}\|x-y\|^2,
\end{equation}
for some $\gamma > 0$. The proximal operator is well-defined since the function $F(\cdot)+\frac{1}{2 \gamma}\|\cdot-y\|^2$ is $1/\gamma$-strongly convex. 

\subsection{Inner Products and Preconditioners}\label{sec: inner product}
The standard $l^{2}$ dot product of Euclidean space $(\cdot,\cdot)$ is usually chosen as the inner product and the norm induced is the Euclidean norm. We now introduce inner product $(\cdot, \cdot)_{\IX}$ induced by a given SPD operator $\IX: \mX\to \mX$ defined as follows
$$
(x, y)_{\IX}:=(\IX x, y) = (x, \IX y), \quad \forall x, y \in \mX
$$ and associated norm $\|\cdot\|_{\IX}$, given by
$\|x\|_{\IX}= (x, x)^{1/2}_{\IX}.$
The dual norm w.r.t the $\IX$-norm is defined as: for $\ell \in \mX^{\prime}$
$$
\|\ell\|_{\mX^{\prime}}=\sup _{0 \neq x \in \mX} \frac{\langle\ell, x\rangle}{\| x \|_{\IX}} =\left( \ell, \ell\right )_{ \IX^{-1}}^{1/ 2} = \left( \IX^{-1} \ell, \ell\right )^{1/ 2}.
$$

We shall generalize the convexity and Lipschitz continuity with respect to $\IX$-norm: we say $F \in \mathcal{S}_{\mu_{ \IX}}$ with $\mu_{\IX} \geqslant 0$ if $F$ is differentiable and
$$
D_{F}(x, y)\geqslant \frac{\mu_{\IX}}{2}\left\|x-y\right\|^{2}_{\IX}, \quad \forall x, y  \in \mX.
$$
In addition, denote $F \in \mathcal{S}_{\mu_{ \IX}, L_{\IX}}$ if $F \in \mathcal{S}_{\mu_{\IX}}$ and there exists $L_{\IX}>0$ such that
$$
D_{F}(x,y) \leq \frac{L_{\IX}}{2}\left\|x-y\right\|^{2}_{\IX}, \quad \forall  x, y   \in \mX.
$$
The gradient method in the inner product $ \IX$ reads as:
\begin{equation}\label{eq:PGD}
x_{k+1} = x_k - \IX^{-1} \nabla F(x_k),
\end{equation}
where $\IX^{-1}$ can be understood as a preconditioner. The convergence rate of the preconditioned gradient descent iteration~\eqref{eq:PGD} is $1-1/\kappa_{\IX}(F)$ with condition number $\kappa_{\IX}(F) = L_{F, \IX}/\mu_{F, \IX}$. 
To simplify notation, we skip $\IX$ in the constants $\mu$ and $ L$, e.g., write $\mu_{F, \IX}$ as $\mu_F$, but keep in the condition number, e.g. $\kappa_{\IX}(F)$, to emphasize that the condition number is measured in the $\IX$ inner product.

For two symmetric operators $A, B: \mX\to \mX$, we say $A \geq (>) B$ if $A  -  B$ is positive semidefinite (definite). Therefore $A> 0 $ means $A$ is SPD. One can easily verify that $A\geq B$ is equivalent to $\lambda_{\min}(B^{-1}A)\geq 1$ or $\lambda_{\max}(A^{-1}B)\leq 1$.

For a general symmetric matrix $A$, we define $$\|x\|_A^2 :=(x,x)_A := x^{\intercal}Ax.$$  In general, $\|\cdot \|_A$ may not be a norm if $A$ is not SPD. However the identity for squares still holds:
\begin{equation}\label{eq:squares}
	2(a,b)_A= \|a\|_A^2 + \|b\|_A^2 - \|a-b\|_A^2.
\end{equation}

\subsection{Lyapunov Analysis}
In order to study the stability of an equilibrium $x^*$ of a dynamical system defined by an autonomous system
\begin{equation}\label{autonomous system}
    x' = \mathcal{G}(x(t)),
\end{equation}
Lyapunov introduced the so-called Lyapunov function $\mathcal{E}(x)$~\cite{Khalil:1173048,Haddad2008}, which is nonnegative and the equilibrium point $x^*$ satisfies $\mathcal{E}\left(x^{*}\right)=0$ and the Lyapunov condition:
$-\nabla \mathcal{E}(x) \cdot \mathcal{G}(x) > 0$ for $x$ near the equilibrium point $x^{*}$. 
Then the (local) decay property of $\mathcal{E}(x)$ along the trajectory $x(t)$ of the autonomous system~\eqref{autonomous system} can be derived immediately:
$$
\frac{\mathrm{d}}{\mathrm{d} t} \mathcal{E}(x(t))=\nabla \mathcal{E}(x) \cdot x^{\prime}(t)=\nabla \mathcal{E}(x) \cdot \mathcal{G}(x)<0.
$$

To further establish the convergence rate of $\mathcal{E}(x(t))$, Chen and Luo~\cite{chen2021unified} introduced the strong Lyapunov condition: $\mathcal{E}(x)$ is a Lyapunov function and there exist constant $q \geqslant 1$, strictly positive function $c(x)$ and function $p(x)$ such that
\begin{equation*}
- \nabla \mathcal{E}(x) \cdot \mathcal{G}(x) \geq c(x) \mathcal{E}^{q}(x)+ p^{2}(x)  
\end{equation*}
holds true near $x^{*}$. From this, one can derive the exponential decay $\mathcal{E}(x(t))=O\left(e^{-c t}\right)$ for $q=1$ and the algebraic decay $\mathcal{E}(x(t))=O\left(t^{- 1 /(q-1)}\right)$ for $q>1$. Furthermore if $\|x - x^*\|^2 \leq C \mathcal E(x)$, then we can derive the exponential stability of $x^*$ from the exponential decay of Lyapunov function $\mathcal E(x)$. 

Note that for an optimization problem, we have freedom to design the flow, i.e., the vector field $\mathcal G(x)$, and Lyapunov function $\mathcal E(x)$.


\section{Gradient and skew-Symmetric Splitting Methods}\label{sec: gradient methods}
	In this section, we consider the generalized gradient flow
	\begin{equation}\label{eq: gradient flow}
x^{\prime}(t) = - \mathcal{A}(x(t)),
\end{equation}
and derive several iterative methods for solving $\mathcal A(x) = 0$ by applying ODE solvers to~\eqref{eq: gradient flow}. Based on the gradient and skew-symmetric splitting~\eqref{eq:dec}, we apply Accelerated OverRelaxation (AOR)~\cite{hadjidimos1978accelerated,hadjidimos2000successive} technique to the non-symmetric part to obtain explicit schemes with linear convergence rate $(1+ c/ \kappa (\mathcal A))^{-1}$.


\subsection{Stability}
Define the quadratic Lyapunov function:
\begin{equation}\label{eq: Eq}
\mathcal{E}_q(x) = \frac{1}{2}\|x-x^{\star}\|^2.
\end{equation}
For $x(t)$ solving the generalized gradient flow~\eqref{eq: gradient flow}:
\begin{equation*}
\begin{aligned}
        \frac{\dd}{\dd t} \mathcal{E}_q(x(t)) &= \langle \nabla \mathcal E_q(x), x' \rangle = -\langle x- x^{\star}, \mathcal{A}(x) - \mathcal A(x^{\star}) \rangle \leq -2\mu \mathcal{E}_q(x),
\end{aligned}
\end{equation*}
which leads to the exponential decay
$$
\mathcal{E}_q(x(t))\leq e^{-2\mu t}\mathcal E_q(x(0))
$$
and consequently the solution $x^*$ is an exponentially stable equilibrium point of the dynamic system defined by~\eqref{eq: gradient flow}.

	
\subsection{Implicit Euler Schemes}\label{sec:implicit scheme}
	The implicit Euler scheme for the generalized gradient flow~\eqref{eq: gradient flow} with step size $\alpha_k>0$ is
\begin{equation}\label{eq:implicit}
	x_{k+1} - x_k = -\alpha_k \mathcal A(x_{k+1}),
\end{equation}
which is equivalent to the proximal iteration 
\begin{equation}\label{eq:proximal}
   (I + \alpha_k \mathcal A) (x_{k+1}) =  x_{k}.
\end{equation}
As $\mathcal A$ is strongly monotone and Liptschitz continuous, the operator $(I + \alpha_k \mathcal A)^{-1}$, which is called the resolvent of $\mathcal A$, is nonexpansive~\cite{browder1967construction} and has a unique fixed point~\cite{rockafellar1976monotone}. 
	
Using the identity of squares and the strong monotonicity, we have
\begin{equation}\label{eq:IEproof}
    	\begin{aligned}
		        \mathcal{E}_q(x_{k+1}) - \mathcal{E}_q(x_{k}) ={}& (x_{k+1}-x_k,x_{k+1}-x^{\star}) -\frac{1}{2}\|x_{k+1}-x_k\|^2\\
		={}& -\alpha_k \dual{ \mathcal A(x_{k+1}) - \mathcal A(x^{\star}), x_{k+1}-x^{\star}} - \frac{1}{2}\|x_{k+1}-x_k\|^2\\
		\leq {}&-2\alpha_k \mu  \mathcal{E}_q(x_{k+1}) ,
	\end{aligned}
\end{equation}
from which the linear convergence follows naturally 
\begin{equation*}
    	\|x_{k+1}-x^{\star}\|^2\leq \frac{1}{1+2\alpha_k \mu}\|x_{k}-x^{\star}\|^2.
\end{equation*}
There is no restriction on the step size $\alpha_k$ which is known as the unconditional stability of the implicit Euler method. Choosing $\alpha_k\gg 1$ will accelerate the convergence with the price of solving a regularized nonlinear problem~\eqref{eq:proximal}, which is in general not practical. The implicit Euler method is presented here as a reference to measure the difference of other and more practical methods.

\subsection{Explicit Euler Schemes}
The explicit Euler scheme for the generalized gradient flow~\eqref{eq: gradient flow} leads to the gradient decent method:
\begin{equation}\label{eq: EE gradient}
    x_{k+1} = x_k - \alpha_k \mathcal{A}(x_k).
\end{equation}
We write~\eqref{eq: EE gradient} as a correction of the implicit Euler scheme~\eqref{eq:implicit}:
$$
x_{k+1}-x_k = - \alpha_k\mathcal A(x_{k+1}) + \alpha_k ( \mathcal A(x_{k+1}) - \mathcal A(x_{k})),
$$
and follow~\eqref{eq:IEproof} to obtain
\begin{equation}\label{eq:EEproof 1}
\begin{aligned}
        \mathcal{E}_q(x_{k+1}) - \mathcal{E}_q(x_{k}) = &(x_{k+1}-x_k,x_{k+1}-x^{\star})- \frac{1}{2}\|x_{k+1} - x_k\|^2  \\
\leq &-2\alpha_k \mu  \mathcal{E}_q(x_{k+1})  - \frac{1}{2}\|x_{k+1} - x_k\|^2\\
&-\alpha_k \dual{ \mathcal A(x_{k}) - \mathcal A(x_{k+1}), x_{k+1}-x^{\star}}.
\end{aligned}
\end{equation}
Here we use the implicit Euler scheme as the reference scheme and call the term $\alpha_k (\mathcal A(x_k) - \mathcal A(x_{k+1}), x_{k+1} - x^{\star})$ in \eqref{eq:EEproof 1} a mis-matching term. Using the Cauchy-Schwarz inequality and the Lipschitz continuity of $\mathcal A$,
\begin{equation}\label{eq:EEproof 2}
\begin{aligned}
&\alpha_k |\dual{ \mathcal A(x_{k}) - \mathcal A(x_{k+1}), x_{k+1}-x^{\star}}| \\
\leq{} &\frac{\alpha_k}{2\mu} \|\mathcal A(x_{k}) - \mathcal A(x_{k+1})\|^2 + \frac{\alpha_k \mu}{2} \|x_{k+1} -x^{\star}\|^2 \\
\leq{} &  \frac{ \alpha_k L_{\mathcal A}^2}{2\mu} \|x_{k+1} -x_k\|^2 + \alpha_k \mu\, \mathcal{E}_q(x_{k+1})
\end{aligned}
\end{equation}
Substitute back to~\eqref{eq:EEproof 1} and let
$\displaystyle \alpha_k = \mu/L_{\mathcal A}^2$, we obtain
$$ \mathcal{E}_q(x_{k+1}) \leq \frac{1}{1+1/\kappa^{2}(\mathcal A)}\mathcal{E}_q(x_{k}).$$
This linear rate is pretty slow when $\kappa(\mathcal A)\gg 1$. We will discuss techniques to improve the dependence $\kappa^2(\mathcal A)$. 


	\subsection{Accelerated Overrelaxation Methods for Shifted skew-Symmetric  Problems}\label{sec:AORlinear}
In this subsection, we consider the shifted skew-symmetric linear system 
	\begin{equation}\label{eq:I+N}
		(\mu I + \mathcal N)x = b,
	\end{equation}
which is a special case of~\eqref{eq:Ax} with $F(x) = \frac{\mu}{2}\|x\|^2 - (b,x)$ and $\mathcal A(x) =  (\mu I + \mathcal N)x - b$. An efficient solver for~\eqref{eq:I+N} is an ingredient of AGSS methods for the nonlinear equation~\eqref{eq:Ax} we shall develop later on. 

The system~\eqref{eq:I+N} can be solved with Krylov subspace methods. For example, minimal residual methods~\cite{idema2007minimal,jiang2007algorithm} based on the Lanczos algorithm require short recurrences of vector storage. The convergence theorem using Chebyshev polynomial~\cite{jiang2007algorithm} shows that the residual converges with a linear rate of $O((1+\mu/\|\mathcal N\|)^{-1})$. Compared with Krylov subspace methods for solving~\eqref{eq:I+N}, our method enjoys a similar linear convergence rate but with a much simpler form and in-space vector storage.

Recall that, for the matrix equation
\begin{equation*}
(D - L - U) x = b,
\end{equation*}
where $D$ a diagonal matrix, and $L$ and $U$ are strictly lower and upper triangular matrix, respectively, the Accelerated OverRelaxation (AOR) method~\cite{hadjidimos1978accelerated,hadjidimos2000successive} with the acceleration parameter $r$ and relaxation factor $\omega > 1$ is in the form 
\begin{equation*}
(D - r L) x_{k+1}=[(1-\omega) D+(\omega-r) L+\omega U] x_{k}+\omega b,
\end{equation*}
which is a two-parameter generalization of the well known successive overrelaxation (SOR) method~\cite{young1954iterative}. We shall apply AOR to \eqref{eq:I+N} with well-designed $r$ and $\omega$. 
	
As $\mathcal N = -\mathcal N^{\intercal}$, all diagonal entries of $\mathcal N$ are zero. Let $B = -{\rm lower}(\mathcal N)$ be the lower triangular part of $\mathcal N$. Then $\mathcal N = B^{\intercal} -B$. Let $B^{\sym} = B + B^{\intercal}$ be a symmetrization of $B$. We have the splitting
\begin{align}
	\label{eq:Bs-B} \mathcal N &= B^{\sym} - 2B,\quad \text{and}\\
	\label{eq:B-Bs} \mathcal N &= 2 B^{\intercal} - B^{\sym}.
\end{align}
As a symmetric matrix, all eigenvalues of $B^{\sym}$ are real. Let $\lambda_1 \leq \cdots \leq \lambda_n$ be the eigenvalues of $B^{\sym}$. Since ${\rm trace } (B^{\sym}) = 0$, we know $\lambda_1 <0$ and $\lambda_n >0$. That is $B^{\sym}$ is symmetric but indefinite. Define $L_{B^{\sym}} =  \| B^{\sym} \| = \max \{|\lambda_1|, \lambda_n\}$ and $\tilde{\kappa}( B^{\sym}) = L_{B^{\sym}} / \mu$. 

		
We discretize the gradient flow~\eqref{eq: gradient flow} by the matrix splitting~\eqref{eq:Bs-B} and treat $B^{\sym}$ as the explicit term:
\begin{equation}\label{eq: GS method y}
	\frac{x_{k+1} - x_k}{\alpha} = - \left (\mu x_{k+1}+B^{\sym}x_k - 2B x_{k+1} - b \right).
\end{equation}
The update of $x_{k+1}$ is equivalent to solve
\begin{equation}\label{eq:AOR}
	\left[ (1+\alpha \mu) I - 2\alpha B \right ]x_{k+1} = x_k - \alpha (B^{\sym}x_k - b),
\end{equation}
which is AOR for solving
\begin{align*}
   (\mu I  - B + B^{\intercal}) x =  b
\end{align*}
 with $r = \frac{2\alpha \mu}{1+\alpha \mu}$ and $\omega =  \frac{\alpha \mu}{1+\alpha \mu}$. The left hand side of~\eqref{eq:AOR} is an invertible lower-triangular matrix and $x_{k+1}$ can be efficiently computed by a forward substitution.
%
We can also use the splitting~\eqref{eq:B-Bs} to get another AOR scheme
\begin{equation}\label{eq: GS method 2}
	\frac{x_{k+1} - x_k}{\alpha} = - \left (\mu x_{k+1} - B^{\sym}x_k + 2B^{\intercal} x_{k+1} - b \right),
\end{equation}
for which $x_{k+1}$ can be efficiently computed by a backward substitution. 

Without loss of generality, we focus on the convergence analysis of scheme~\eqref{eq: GS method y}. Notice that when $0<\alpha < 1/L_{B^{\sym}}$, $I \pm \alpha B^{\sym}$ is SPD. 					
For $0<\alpha<1/L_{B^{\sym}}$, consider the Lyapunov function
\begin{equation}\label{eq: linear refined Lyapunov}
	\mathcal{E}^{\alpha B}(x) = \frac{1}{2}\|x-x^{\star}\|^2 -  \frac{\alpha}{2}\| x - x^{\star}\|^2_{B^{\sym}}= \frac{1}{2} \| x- x^{\star} \|^2_{ I - \alpha B^{\sym}} ,
\end{equation}
which is nonnegative and $\mathcal E^{\alpha B}(x)= 0$  if and only if $x = x^{\star}$. Noted that $\mathcal E^{\alpha B}$ might be negative without control on the step size $\alpha$ as $B^{\sym}$ is indefinite. 
					

\begin{theorem}\label{thm: linear convergence for the shifted skew-symmetric linear system}
	Let $\{x_k\}$ be the sequence generated by ~\eqref{eq: GS method y} with arbitrary initial guess $x_0$ and step size $0< \alpha < 1/L_{B^{\sym}}$. Then for the Lyapunov function~\eqref{eq: linear refined Lyapunov}, we have
	\begin{equation}\label{eq:linearAORrate}
		\begin{aligned}
			\mathcal{E}^{\alpha B}( x_{k+1}) \leq&~ \frac{1}{1+\alpha \mu}\mathcal{E}^{\alpha B}(x_{k}).
		\end{aligned}
	\end{equation}
In particular, for $ \alpha = \displaystyle \frac{1}{2L_{B^{\sym}}}$, we have
\begin{equation}\label{eq:AORnormrate}
\| x_k - x^{\star}\|^2 \leq \left (\frac{1}{1 + 1/\tilde{\kappa}(B^{\sym})}\right )^k 3\| x_0 - x^{\star}\|^2.
\end{equation}
\end{theorem}
\begin{proof}
%
We use the identity for squares~\eqref{eq:squares}:
\begin{equation}\label{eq:Eqdiff}
	\frac{1 }{2} \| x_{k+1} - x^{\star}\|^2-\frac{1 }{2} \| x_{k} - x^{\star}\|^2 =  (x_{k+1} -x^{\star}, x_{k+1} - x_k) -\frac{1 }{2} \| x_{k+1} - x_k\|^2. 
\end{equation}
Write GSS scheme~\eqref{eq: GS method y} as a correction of the implicit Euler scheme
\begin{align*}
	x_{k+1} - x_k & = - \alpha (\mathcal A ( x_{k+1}) -\mathcal A(x^{\star}))+ \alpha B^{\sym}(x_{k+1} - x_k).
\end{align*}
For the first term, we have
$$
-\alpha \langle  x_{k+1} -x^{\star},  \mathcal A ( x_{k+1}) -\mathcal A ( x^{\star})  \rangle = -  \alpha \mu \|x_{k+1} -x^{\star}\|^2. 
$$
We use the identity~\eqref{eq:squares} to expand the cross term as
\begin{align*}
(  x_{k+1} -x^{\star}, x_{k+1} - x_k)_{\alpha B^{\sym}} ={} &\frac{1}{2} \| x_{k+1} -x^{\star} \|_{\alpha B^{\sym}}^2 +  \frac{1}{2} \| x_{k+1} -x_k \|_{\alpha B^{\sym}}^2 \\
-&\frac{1}{2} \| x_{k} -x^{\star} \|_{\alpha B^{\sym}}^2.
\end{align*}
Substitute back to~\eqref{eq:Eqdiff} and rearrange the terms, we obtain the identity 
\begin{align*}
{}&	\mathcal E^{\alpha B}(x_{k+1}) - \mathcal E^{\alpha B}(x_{k}) \\
={}& -\alpha \mu \|x_{k+1} -x^{\star}\|^2-\frac{1}{2}\| x_{k+1}- x_k \|_{ I - \alpha B^{\sym}}^2\\
	={}& - \alpha  \mu \, \mathcal E^{\alpha B}(x_{k+1}) -\frac{\alpha \mu}{2}\| x_{k+1}- x^{\star} \|_{ I + \alpha B^{\sym}}^2-\frac{1}{2}\| x_{k+1}- x_k \|_{ I - \alpha B^{\sym}}^2.
\end{align*}
As $0< \alpha < \displaystyle 1/L_{B^{\sym}}$, $I \pm \alpha B^{\sym}$ is SPD and the last two terms are non-positive. Dropping them, we obtain the inequality
$$
\mathcal E^{\alpha B}(x_{k+1}) - \mathcal E^{\alpha B}(x_{k}) \leq - \alpha \mu\, \mathcal E^{\alpha B}(x_{k+1}) 
$$
and~\eqref{eq:linearAORrate} follows by arrangement.
						
When $\alpha = 1/(2L_{B^{\sym}})$, we have the bound	
\begin{equation}
\frac{1}{4}\|x-x^{\star}\|^2 \leq \mathcal E^{\alpha B}(x) \leq \frac{3}{4}\|x-x^{\star}\|^2
\end{equation}
which implies~\eqref{eq:AORnormrate}. 
\end{proof}

For SOR type iterative methods, usually spectral analysis~\cite{bai2007successive,bai2005generalized,hadjidimos2000successive} is applied to the error matrix which is in general non-symmetric and harder to estimate. The Lyapunov analysis provides a new and relatively simple approach and more importantly enables us to study nonlinear systems.


\subsection{Gradient and skew-Symmetric Splitting Methods for Nonlinear Problems}\label{sec: AOR for nonlinear problems}
For non-linear equation~\eqref{eq:Ax}, we treat $\nabla F$ explicitly and propose the following Gradient and skew-Symmetric Splitting (GSS) scheme
\begin{equation}\label{eq: nonlinear SOR 1}
    \frac{x_{k+1} - x_k}{\alpha} = -\left (\nabla F(x_k)+ B^{\sym}x_k - 2B x_{k+1}\right).
\end{equation}
Similarly, modification on~\eqref{eq: GS method 2} gives another GSS scheme
\begin{equation}\label{eq: nonlinear SOR 2}
    \frac{x_{k+1} - x_k}{\alpha} = - \left (\nabla F(x_k) - B^{\sym}x_k + 2B^{\intercal} x_{k+1} \right).
\end{equation}
Both schemes are explicit as $\nabla F(x_k)$ is known and $x_{k+1}$ can be computed by solving a triangular matrix equation.  

We focus on the GSS method~\eqref{eq: nonlinear SOR 1}. The proof for~\eqref{eq: nonlinear SOR 2} follows in line with a sign change in the following Lyapunov function
 \begin{equation}\label{eq: modified nolienar Lyapunov}
    \mathcal E^{\alpha BD}(x_k):= \frac{1}{2}\|x_k-x^{\star}\|^2_{I - \alpha B^{\sym}}  -\alpha D_F(x^{\star}, x_k).
\end{equation}
%

\begin{lemma}\label{lem:positivity of GSS Lyapunov}
For  $F \in \mathcal S_{\mu, L_F}$ and $\alpha  <1/\max\{2L_{B^{\sym}}, 2L_F \}$, the Lyapunov function $\mathcal E^{\alpha BD}(x)\geq 0$ and $\mathcal E^{\alpha BD}(x)= 0$  if and only if $x = x^{\star}$.
\end{lemma}
\begin{proof}
Since   $F \in \mathcal S_{\mu, L}$, the Bregman divergence is non-negative and
\begin{equation*}
  D_F(x^{\star}, x) \leq \frac{L_F}{2}\|x-x^{ \star}\|^2.
\end{equation*}
Then for  $\alpha  <1/\max\{2L_{B^{\sym}}, 2L_F \}$,
\begin{equation*}
\begin{aligned}
     \mathcal{E}^{\alpha BD}(x) &\geq  \frac{1}{2}\|x-x^{\star}\|^2 - \frac{\alpha}{2} \|x-x^{\star}\|_{B^{\sym}} - \frac{\alpha L_F}{2}\|x-x^{ \star}\|^2 \\
    &\geq \frac{1}{2}\|x-x^{\star}\|^2   - \frac{1}{4L_{B^{\sym}}} \|x-x^{\star}\|_{B^{\sym}}^2  - \frac{ 1}{4}\|x-x^{ \star}\|^2  \geq 0.
\end{aligned}
\end{equation*}
If $x \neq x^{\star}$, then the second $\geq$ becomes $>$. So $\mathcal E^{\alpha BD}(x)= 0$  if and only if $x = x^{\star}$. 
\end{proof}

\begin{theorem}\label{thm: linear convergence for generalized nonlinear system}
Let $\{x_k\}$ be the sequence generated by the AOR method~\eqref{eq: nonlinear SOR 1} with arbitrary initial guess $x_0$ and step size $\alpha < 1/\max\{2L_{B^{\sym}}, 2L_F\}$. 
Then for the discrete Lyapunov function~\eqref{eq: modified nolienar Lyapunov}, we have 
\begin{equation*}
\begin{aligned}
\mathcal{E}^{\alpha BD}(x_{k+1})\leq&~ \frac{1}{1+ \alpha \mu}\mathcal{E}^{\alpha BD}(x_k).
\end{aligned}
\end{equation*}
In particular, for $ \alpha = \displaystyle 1/ (4\max \left\{ L_{B^{\sym}}, L_F \right\})$, we have
\begin{equation}\label{eq: nonlinear AOR norm rate}
\| x_k - x^{\star}\|^2 \leq \left (1 + 1/\max\left \{ 4\tilde{\kappa}( B^{\sym}), 4 \kappa(F)\right\}\right )^{-k} 6\| x_0 - x^{\star}\|^2.
\end{equation}
\end{theorem}
\begin{proof}
We write the scheme \eqref{eq: nonlinear SOR 1} as a correction of the implicit Euler scheme
\begin{align*}
x_{k+1} - x_k  = &- \alpha (\mathcal A ( x_{k+1}) -\mathcal A(x^{\star}))+ \alpha B^{\sym}(x_{k+1} - x_k) \\
&+ \alpha(\nabla F(x_{k+1}) - \nabla F(x_k)).
\end{align*}
The first two terms are treated as before; see the proof of Theorem \ref{thm: linear convergence for the shifted skew-symmetric linear system}.
The last cross term is expanded using the identity~\eqref{eq: Bregman divergence identity} for the Bregman divergence:
\begin{equation*}
\begin{aligned}
   \langle  x_{k+1} - x^{\star}, \nabla F(x_{x+1}) - \nabla F(x_k) \rangle  = D_F(x^{\star}, x_{k+1}) + D_F( x_{k+1}, x_k) - D_F( x^{\star}, x_k).
\end{aligned}
\end{equation*}
Substituting back to~\eqref{eq:Eqdiff} and rearranging the terms, we obtain the identity
\begin{align*}
\mathcal E^{\alpha BD}(x_{k+1}) - \mathcal E^{\alpha BD}(x_{k}) =&~ - \alpha \mu \mathcal E^{\alpha BD}(x_{k+1}) - \frac{\alpha \mu}{2} \|x_{k+1} -x^{\star}\|_{I + \alpha B^{\sym}}^2\\
&- \alpha^2 \mu D_F(x^{\star}, x_{k+1})\\
& - \frac{1}{2}\| x_{k+1}- x_k \|_{I -\alpha B^{\sym}}^2+ \alpha D_F( x_{k+1}, x_k).
\end{align*}
When $\alpha  < 1/\max\{2L_{B^{\sym}}, 2L_F \}$,  $I \pm\alpha B^{\sym}$ is SPD and
\begin{equation*}
\begin{aligned}
  \alpha D_F( x_{k+1},  x_k)&\leq \frac{\alpha L}{2}\|x_{k+1}- x_k\|^2 \leq  \frac{1}{4}\|x_{k+1}- x_k\|^2\leq \frac{1}{2}\| x_{k+1}- x_k \|_{I -\alpha B^{\sym}}^2.
\end{aligned}
\end{equation*} 
Thus the last two terms are non-positive. Dropping all non-positive terms, we obtain the inequality
$$
\mathcal E^{\alpha BD}(x_{k+1}) - \mathcal E^{\alpha BD}(x_{k}) \leq - \alpha \mu \mathcal E^{\alpha BD}(x_{k+1}) 
$$
and the linear reduction follows by arrangement.
						
When $\alpha = \displaystyle \min \left\{\frac{1}{4L_{B^{\sym}}}, \frac{1}{4L} \right\}$, we have the bound	
\begin{equation}
\frac{1}{8}\|x-x^{\star}\|^2 \leq \mathcal E^{\alpha BD}(x) \leq \frac{3}{4}\|x-x^{\star}\|^2
\end{equation}
which implies~\eqref{eq: nonlinear AOR norm rate}.
\end{proof}

Notice the rate in Theorem \ref{thm: linear convergence for generalized nonlinear system} is $(1+c/\kappa)$ for $\kappa =  \max\{\kappa(F),\tilde{\kappa}(B^{\sym})\}$ which matches the rate of the gradient descent method for convex optimization problems. We expect a combination of the accelerated gradient flow and AOR for the skew-symmetric part will give an accelerated explicit scheme. 

\subsection{Implicit-Explicit Euler Schemes}
%
%
If we treated the skew-symmetric part implicitly, we have the implicit-explicit scheme: 
\begin{equation}\label{eq: semi-implicit Euler}
    x_{k+1} = x_k - \alpha_k \nabla F(x_k) - \alpha_k \mathcal N(x_{k+1}).
\end{equation}
Define the Lyapunov function 
\begin{equation}\label{eq: semi-implicit Lyapunov}
    \mathcal E^{\alpha D}(x_k) = \frac{1}{2}\|x_k - x^{\star}\| - \alpha D_F( x^\star, x_k).
\end{equation}
Following Lemma \ref{lem:positivity of GSS Lyapunov}, $\mathcal E^{\alpha D} \geq 0$ if $\alpha < 1/L_F$ and $\mathcal E^{\alpha D} =0$ iff $x_k = x^{\star}$.

Combining the proof of Theorem \ref{thm: linear convergence for the shifted skew-symmetric linear system} and Theorem \ref{thm: linear convergence for generalized nonlinear system}, we obtain the following linear convergence.
\begin{theorem}\label{thm:IMEX}
Let $\{x_k\}$ be the sequence generated by the scheme~\eqref{eq: semi-implicit Euler} with arbitrary initial guess $x_0$ and step size $\alpha < 1/L_F$. 
Then for the discrete Lyapunov function~\eqref{eq: semi-implicit Lyapunov}, we have 
\begin{equation*}
\begin{aligned}
\mathcal{E}^{\alpha D}(x_{k+1})\leq&~ \frac{1}{1+ \alpha \mu}\mathcal{E}^{\alpha D}(x_k).
\end{aligned}
\end{equation*}
In particular, for $ \alpha = \displaystyle \frac{1}{2L_{F}}$, we have
\begin{equation}\label{eq: semi-implicit norm rate}
\| x_k - x^{\star}\|^2 \leq \left (1 + 1/ (2 \kappa(F)\right )^{-k} 2\| x_0 - x^{\star}\|^2.
\end{equation}
\end{theorem}

With the explicit part as the gradient of a convex function, we can accommodate acceleration techniques~\cite{luoDifferentialEquationSolvers2021a} for convex optimization to achieve accelerated rates; see the proposed AGSS scheme in Section \ref{sec: IMEX scheme}.

%



\section{Accelerated Gradient and skew-Symmetric Splitting Methods}\label{sec:AGSS flow and schemes}
In this section, we shall develop the accelerated gradient flow and propose Accelerated Gradient and Skew-symmetric Splitting (AGSS) methods with accelerated linear rates. For the implicit-explicit scheme, we can relax to inexact inner solvers with computable error tolerance. 

\subsection{Accelerated Gradient Flow}
We introduce an auxiliary variable $y$ and an accelerated gradient flow
\begin{equation}\label{eq:AG}
\left \{\begin{aligned}
     x^{\prime} &= y - x ,\\
     y^{\prime} & = x - y -\mu^{-1}(\nabla F(x) + \mathcal N y).
\end{aligned}\right .
\end{equation}
Comparing with the accelerated flow in ~\cite{luoDifferentialEquationSolvers2021a} for convex optimization, the difference is to split $\mathcal A(x) \rightarrow \nabla F(x) + \mathcal N y$. Denote the vector field on the right hand side of~\eqref{eq:AG} by $\mathcal G(x,y)$. Then $\mathcal G(x^{\star},x^{\star}) = 0$ and thus $(x^{\star},x^{\star})$ is an equilibrium point of~\eqref{eq:AG}. 
 
We first show $(x^{\star},x^{\star})$ is exponentially stable. Consider the Lyapunov function:
\begin{equation}\label{eq: acc Lyapunov, mixed}
\mathcal{E}(x, y) =  D_F(x, x^{\star}) + \frac{\mu}{2}\| y-x^{\star}\|^2.
\end{equation}
For $\mu$-strongly convex $F$, function $D_F(\cdot , x^{\star}) \in \mathcal S_{\mu}$. Then $\mathcal{E}(x, y)\geq 0$ and $\mathcal{E}(x, y)= 0$ iff $x=y=x^{\star}$.

\begin{lemma}
For function $F \in \mathcal S_{\mu}$, we have
\begin{equation}\label{eq:DFmu}
 \langle \nabla F(x) - \nabla F(x^{\star}), x- x^{\star} \rangle \geq  D_F(x, x^{\star}) + \frac{\mu}{2}\|x - x^{\star}\|^2,
\end{equation}
\end{lemma}
\begin{proof}
By direct computation, $ \langle \nabla F(x) - \nabla F(x^{\star}), x- x^{\star} \rangle = D_F(x,x^{\star}) + D_F(x^{\star}, x)$ and thus~\eqref{eq:DFmu} follows from the bound
$$\min \{ D_F(x,x^{\star}), D_F(x^{\star},x)\} \geq \frac{\mu}{2}\|x - x^{\star}\|^2,$$
for a convex function $F \in \mathcal S_{\mu}$.
\end{proof}

We then verify the strong Lyapunov property.

\begin{lemma}[Strong Lyapunov Property]\label{lem: acc strong Lyapunov}
Assume function $F \in \mathcal{S}_{\mu} $. Then for the Lyapunov function~\eqref{eq: acc Lyapunov, mixed} and the  accelerated gradient flow vector field $\mathcal G$, the following strong Lyapunov property holds
\begin{equation}\label{eq:strong acc}
- \nabla \mathcal E(x,y)\cdot \mathcal G(x,y) \geq \mathcal E(x,y)+\frac{\mu}{2}\|y-x\|^2.
\end{equation}
\end{lemma}
\begin{proof}
First of all, as $\mathcal G(x^{\star}, x^{\star}) = 0$, $$- \nabla \mathcal E(x,y)\cdot \mathcal G(x,y) = - \nabla \mathcal E(x,y)\cdot (\mathcal G(x,y)  - \mathcal G(x^{\star}, x^{\star})).$$
Direct computation gives
\begin{align*}
- \nabla \mathcal E(x,y)\cdot \mathcal G(x,y) =&~
\langle  \nabla D_F(x; x^{\star}),x -x^{\star}-( y - x^{\star}) \rangle  -\mu ( y-x^{\star}, x -x^{\star})\\
 + \mu \|y  - x^{\star}\|^2 & + \langle \nabla F(x) - \nabla F(x^{\star}),y - x^{\star} \rangle + ( y-x^{\star}, \mathcal N (y-x^{\star}))\\
=~ \langle  \nabla F(x) &- \nabla F(x^{\star}),x -x^{\star} \rangle  + \mu \|y  - x^{\star}\|^2-\mu(y-x^{\star},x -x^{\star})
\end{align*}
where we have used $ \nabla D_F(x; x^{\star}) = \nabla F(x) - \nabla F(x^{\star})$ and $( y-x^{\star}, \mathcal N (y-x^{\star})) = 0$ since $\mathcal N$ is skew-symmetric.

Using the identity for squares~\eqref{eq:squares}, we expand
\begin{equation*}
    \begin{aligned}
    &\frac{1}{2} \|y  - x^{\star}\|^2 - ( y-x^{\star}, x -x^{\star})  = \frac{1}{2}\|y-x\|^2-\frac{1}{2}\|x - x^{\star}\|^2.
    \end{aligned}
\end{equation*}
Under the bound~\eqref{eq:DFmu},  we obtain~\eqref{eq:strong acc}.
\end{proof}
%

The calculation is more clear when $\nabla F(x) = Ax$ is linear with $A = \nabla^2 F \geq \mu I$. Then $- \nabla \mathcal E(x,y)\cdot (\mathcal G(x,y) - \mathcal G(x^{\star},y^{\star}))$ is a quadratic form $v^{\intercal}Mv$ with $v= (x - x^{\star}, y -x^{\star})^{\intercal}$. We calculate the matrix $M$ as
$$
\begin{pmatrix}
A & 0 \\
0  & \mu I
\end{pmatrix}
\begin{pmatrix}
  I & - I \\
 - I + \mu^{-1} A &   \quad I+ \mu^{-1} \mathcal N
\end{pmatrix}
=
\begin{pmatrix}
A & - A \\
 - \mu I + A & \;  \mu I + \mathcal N
\end{pmatrix}.
$$
As $v^{\intercal}Mv = v^{\intercal}\sym(M)v$, we compute the symmetric part
\begin{align*}
\sym
\begin{pmatrix}
A & -A\\
 -\mu I +A & \;   \mu I + \mathcal N
\end{pmatrix}
&=
\begin{pmatrix}
A & - \mu I/2 \\
 - \mu I/2 & \mu I
\end{pmatrix}\\
&\geq
\begin{pmatrix}
A/2 & 0 \\
0 & \;  \mu I/2
\end{pmatrix} + \frac{1}{2}
\begin{pmatrix}
\mu I & - \mu I \\
 -\mu I & \;  \mu I
\end{pmatrix},
\end{align*}
where in the last step we use the convexity $A\geq \mu I$. Then~\eqref{eq:strong acc} follows.

To see how the condition number changes using the accelerated gradient flow, we consider the following $2\times 2$ matrix
\begin{equation}
G =
\begin{pmatrix}
-1 & 1\\
1- a & \;  -1 + b{\rm i}
\end{pmatrix}
\end{equation}
with $a\geq 1$ representing the eigenvalue of $\nabla^2F/\mu$ and $b{\rm i}$ for $\mathcal N/\mu$ as the eigenvalue of skew-symmetric matrix is pure imaginary. Then the eigenvalue of $G$ is
\begin{equation*}
    \lambda (G)  = -1  + \frac{b \pm \sqrt{b^2+4(a-1)}}{2}{\rm i}.
\end{equation*}
The real part is always $-1$ which implies the decay property of ODE $x'=Gx$. The spectral radius is 
$$|\lambda|= \mathcal  O(\sqrt{a}) + \mathcal O(|b|),  \quad \text{as } a\gg 1, |b|\gg 1. $$
As a comparison, $| a + b {\rm i} | =\sqrt{a^2 + b^2} = \mathcal O(a) + \mathcal  O(|b|)$. We improve the dependence from $\mathcal  O(a)$ to $\mathcal O(\sqrt{a})$.

\subsection{Implicit Euler Schemes}
If we apply the implicit Euler method for the accelerated gradient system~\eqref{eq:AG}, the linear convergence is a direct consequence of the strong Lyapunov property~\eqref{eq:strong acc}. Consider
\begin{subequations}
\begin{align}
\label{eq:AGIE1}     \frac{x_{k+1} - x_k}{\alpha_k} &= \mathcal G^x(x_{k+1}, y_{k+1}):=y_{k+1} - x_{k+1},\\
\label{eq:AGIE2}      \frac{y_{k+1} - y_k}{\alpha_k} & = \mathcal G^y(x_{k+1}, y_{k+1}) :=x_{k+1} - y_{k+1} - \mu^{-1}(\nabla F(x_{k+1}) + \mathcal N y_{k+1}).
\end{align}
\end{subequations}
As we have shown before, the implicit Euler method is unconditionally stable with the price of solving a nonlinear equation system where $x_{k+1}$ and $y_{k+1}$ are coupled together, which is usually not practical. Again we present the convergence analysis here as a reference.

Denote $\mathcal E_{k} = \mathcal E(x_{k}, y_{k})$ and $z
_k = (x_k , y_k)$. Using the convexity of the Lyapunov function and under the same assumption of Lemma \ref{lem: acc strong Lyapunov}, we have
\begin{equation*}
    	\begin{aligned}
\mathcal E_{k+1} - \mathcal E_{k}\leq{}& (\nabla \mathcal E_{k+1},z_{k+1}-z_k) -\frac{\mu}{2}\|z_{k+1}-z_k\|^2\\
		={}& \alpha_k \dual{ \mathcal \nabla \mathcal E_{k+1}, \mathcal G(z_{k+1})} - \frac{\mu}{2}\|z_{k+1}-z_k\|^2\\
		\leq {}&-\alpha_k \mathcal{E}_{k+1} ,
	\end{aligned}
\end{equation*}
from which the linear convergence follows naturally
\begin{equation*}
    	\mathcal E_{k+1}\leq \frac{1}{1+\alpha_k }\mathcal E_{k}
\end{equation*}
for arbitrary step size $\alpha_k>0$.

\subsection{Implicit-Explicit Scheme}\label{sec: IMEX scheme}
If we treat the skew symmetric part implicit, we can achieve the acceleration like the convex optimization problem. Consider the following implicit-explicit (IMEX) scheme of the accelerated gradient flow:
\begin{subequations}\label{eq:semi-implicit}
\begin{align}
\label{eq:imex1}    \frac{\hat{x}_{k+1}-x_k}{\alpha_k} &=  y_k - \hat{x}_{k+1},  \\
\label{eq:imex2}    \frac{y_{k+1}-y_k}{\alpha_k} &=  \hat{x}_{k+1} - y_{k+1} -\mu^{-1} \left( \nabla F(\hat{x}_{k+1}) + \mathcal Ny_{k+1}\right) , \\
\label{eq:imex3}    \frac{x_{k+1}-x_k}{\alpha_k} &= y_{k+1} - x_{k+1}.
\end{align}
\end{subequations}
 We first treat $y$ known as $y_k$ and solve the first equation to get $\hat{x}_{k+1}$ and then with known $\hat{x}_{k+1}$ to solve the following linear equation
\begin{equation}\label{eq:I+N 2}
\left [ (1+\alpha_k)I + \alpha_k\mu^{-1} \mathcal N \right ] y_{k+1} = b(\hat{x}_{k+1},y_k),
\end{equation}
with known right hand side $b(\hat{x}_{k+1},y_k) = \alpha_k \hat x_{k+1} - \alpha_k\mu^{-1} \nabla F(\hat x_{k+1}) +  y_k $. The linear equation is associated to a shifted skew-symmetric system $(\beta  I + \mathcal N)$ with $\beta = 1 +\mu/\alpha_k$, which can be solved as discussion in Section \ref{sec:AORlinear}. 
Then with computed $y_{k+1}$, we can solve for $x_{k+1}$ again using an implicit discretization of $x^{\prime} = y - x$. In terms of the ODE solvers,~\eqref{eq:semi-implicit} is known as the predictor-corrector method. The intermediate approximation $\hat{x}_{k+1}$ is a predictor and $x_{k+1}$ is a corrector.

As the skew-symmetric part is treated implicitly, the scheme is expected to achieve the accelerated linear rate as the accelerated gradient method for minimizing the convex function $F$.
For simplicity, we denote 
\begin{equation}\label{eq:hatE}
\mathcal{E}_{k} = \mathcal{E}(x_k, y_k),\quad \hat{\mathcal{E}}_{k+1} = \mathcal{E}(\hat x_{k+1}, y_{k+1})
\end{equation}
for the Lyapunov function $\mathcal{E}$ defined in~\eqref{eq: acc Lyapunov, mixed}.

\begin{lemma}\label{lem: acc gra method semi-implicit decay}
Assume function $F \in \mathcal{S}_{\mu}$. Let $(\hat x_k, y_k)$ be the sequence generated by the accelerated gradient method~\eqref{eq:semi-implicit}. Then for $k \geq 0$,
\begin{equation}\label{eq:hatdecay}
\begin{aligned}
\hat{\mathcal{E}}_{k+1} - \mathcal{E}_k \leq&~ -\alpha_k \hat{\mathcal{E}}_{k+1} -\alpha_k  \left\langle \nabla D_F(\hat x_{k+1}, x^{\star}) , y_{k+1}-y_k\right \rangle   -\frac{\mu}{2}\left\|y_{k+1}-y_k\right\|^2 \\
&
 -\frac{\alpha_k \mu}{2}\|y_{k+1}-\hat x_{k+1}\|^2-D_F(\hat x_{k+1}, x_k) . 
\end{aligned}
\end{equation}
\end{lemma}
\begin{proof}
Direct computation using \eqref{eq: Bregman divergence identity} and \eqref{eq:squares} gives
 \begin{align*}
\hat{\mathcal{E}}_{k+1} - \mathcal{E}_k  ={}& \langle \partial_x \mathcal E(\hat{x}_{k+1}, y_{k+1}), \hat{x}_{k+1} - x_k \rangle + \langle \partial_y \mathcal E(\hat{x}_{k+1}, y_{k+1}), y_{k+1} - y_k \rangle\\
& -D_F(\hat x_{k+1}, x_k)  -\frac{\mu}{2}\left\|y_{k+1}-y_k\right\|^2.
\end{align*}
Then substitute
\begin{align*}
\hat{x}_{k+1} - x_k &= \alpha_k \mathcal G^x (\hat{x}_{k+1}, y_{k+1}) + \alpha_k (y_k - y_{k+1})\\
y_{k+1} - y_k & = \alpha_k \mathcal G^y (\hat{x}_{k+1}, y_{k+1}),
\end{align*}
and use the strong Lyapunov property~\eqref{eq:strong acc} at $(\hat{x}_{k+1}, y_{k+1})$ to get the desired result.
\end{proof}

The term $-\alpha_k \left\langle \nabla D_F(\hat x_{k+1}; x^{\star}) , y_{k+1}-y_k\right \rangle$ on the right hand side of~\eqref{eq:hatdecay} accounts for $-\alpha_k  \langle \partial_x \mathcal E(\hat{x}_{k+1}, y_{k+1}), y_{k+1} - y_k \rangle$ for using explicit $y_k$ in~\eqref{eq:imex1}, which is again a mis-match term compare with $\mathcal G(\hat{x}_{k+1}, y_{k+1})$ used in the implicit Euler scheme~\eqref{eq:AGIE1}.
The correction step~\eqref{eq:imex3} will be used to bound the mis-match term. We restate and generalize the result in ~\cite{luoDifferentialEquationSolvers2021a} in the following lemma.

\begin{lemma}\label{lem:correction step bound}
Assume $\nabla F$ is $L_F$-Lipschitz continuous and $(\hat x_{k+1},y_{k+1})$ is generated from $(x_k, y_k)$ such that: for the Lyapunov function \eqref{eq: acc Lyapunov, mixed}, there exists $c_1, c_2, c_3 >0$ satisfying 
\begin{equation}\label{eq:hatLyapunovdiff}
 \hat{\mathcal{E}}_{k+1} - \mathcal{E}_k \leq~ -\alpha_k c_1  \hat{\mathcal{E}}_{k+1} -\alpha_k c_2\left\langle \nabla D_F(\hat x_{k+1}, x^{\star}), y_{k+1}-y_k\right \rangle -\frac{c_3}{2}\left\|y_{k+1}-y_k\right\|_{\IX}^2. 
\end{equation}
Then set $x_{k+1}$ by the relation
\begin{equation}\label{eq:extrapolation}
(1+\alpha_k c_1)(x_{k+1} - \hat x_{k+1}) = \alpha_k c_2 (   y_{k+1} - y_k ). 
\end{equation}
and choose $\alpha_k>0$ satisfying
\begin{equation}\label{eq:alpha}
\alpha_k^2  L_Fc_2^2 \leq (1+\alpha_k c_1)c_3,
\end{equation}
we have the linear convergence
\begin{equation*}
    \mathcal{E}_{k+1} \leq \frac{1}{1+\alpha_k c_1} \mathcal{E}_k.
\end{equation*}
\end{lemma}
\begin{proof}
Using the relation~\eqref{eq:extrapolation}, we rewrite~\eqref{eq:hatLyapunovdiff} into
\begin{equation}\label{eq:hatLyapunovdiff new}
\begin{aligned}
    &(1+\alpha_k c_1)\hat{\mathcal{E}}_{k+1} - \mathcal{E}_k \\
\leq&~  -(1+c_1\alpha_k)\left\langle  \nabla D_F(\hat x_{k+1}, x^{\star}), x_{k+1} - \hat x_{k+1}\right\rangle -\frac{(1+\alpha_k c_1)^2c_3}{2\alpha_k^2 c_2^2}\left\|x_{k+1} - \hat x_{k+1}\right\|^2.
\end{aligned}
\end{equation}
As $D_F(\cdot, x^{\star}) \in \mathcal S_{\mu,L_F}$, we have
\begin{equation}\label{eq: correction decay}
\begin{aligned}
        \mathcal{E}_{k+1} - \hat{\mathcal{E}}_{k+1} &= D_F(x_{k+1}, x^{\star}) - D_F(\hat{x}_{k+1}, x^{\star}) \\
&        \leq \left\langle \nabla D_F(\hat x_{k+1}, x^{\star}),x_{k+1} - \hat x_{k+1}\right\rangle + \frac{L_F}{2}\|x_{k+1} - \hat x_{k+1}\|^2.
\end{aligned}
\end{equation}

Summing $(1+c_1\alpha_k)\times$\eqref{eq: correction decay} and~\eqref{eq:hatLyapunovdiff new} we get
\begin{equation*}
    \begin{aligned}
         (1+c_1\alpha_k) \mathcal{E}_{k+1} - \mathcal{E}_k \leq -\left(\frac{(1+\alpha_k c_1)^2c_3}{2\alpha_k^2 c_2^2  }-\frac{(1+\alpha_k c_1)L_F}{2 }\right)\left\|x_{k+1} - \hat x_{k+1}\right\|^2 \leq 0
    \end{aligned}
\end{equation*}
according to our choice of $\alpha_k$. Rearrange the terms to get the desired inequality.
\end{proof}

\begin{remark}\rm 
The largest step size $\alpha_k$ satisfying~\eqref{eq:alpha} is given by
$$
\alpha_k = \frac{c_1c_3+\sqrt{c_1^2c_3^2 + 4L_F c_2^2c_3}}{2L_F c_2^2} \geq \frac{1}{c_2}\sqrt{\frac{c_3}{L_F}},
$$
where the last one is a simplified formula for the step size satisfying~\eqref{eq:alpha}.
\end{remark}

We showed the decay of Lyapunov function in Lemma \ref{lem: acc gra method semi-implicit decay} satisfying the assumption of Lemma \ref{lem:correction step bound} with $c_1=c_2= 1, c_3 = \mu$ and the correction step~\eqref{eq:imex3} matches the relation~\eqref{eq:extrapolation}. As a result, we state the following accelerated linear convergence rate of the IMEX scheme~\eqref{eq:semi-implicit}.

\begin{theorem}\label{thm: linear convergence of AGSS IMEX scheme}
Assume function $F \in \mathcal{S}_{\mu, L_F} $. Let $(x_k, y_k)$ be the sequence generated by the accelerated gradient method~\eqref{eq:semi-implicit} with arbitrary initial value and step size satisfying 
$$ 
\alpha_k^2 L_F \leq (1+\alpha_k) \mu,
$$ 
then for the Lyapunov function~\eqref{eq: acc Lyapunov, mixed},
\begin{equation*}
    \mathcal{E}_{k+1} \leq \frac{1}{1+\alpha_k } \mathcal{E}_k.
\end{equation*}
For $ \alpha_k =1/\sqrt{\kappa(F)}$, we achieve the accelerated rate
\begin{equation*}
    \mathcal{E}_k \leq \left( \frac{1}{1+1/\sqrt{\kappa(F)}} \right)^{k} \mathcal{E}_0,
\end{equation*}
which implies
\begin{equation*}
\|x_{k}-x^{\star}\|^2 + \|y_{k}-x^{\star}\|^2 \leq \left (\frac{1}{1+1/\sqrt{\kappa(F)} }\right )^{k} 2 \mathcal{E}_0/\mu.
\end{equation*}
\end{theorem}

For linear problems, in comparison with the HSS scheme~\eqref{eq:HSS}, we achieve the same accelerated rate for both linear and nonlinear systems without the need to compute the inverse of the symmetric part $(\beta I + \nabla^2 F)^{-1}$. The efficiency of our algorithm is confirmed through computations on a convection-diffusion model in Section \ref{sec: Convection-diffusion model}. When relaxing to an inexact inner solver for the skew-symmetric part, our algorithm surpasses HSS methods. We provide an analysis using perturbation arguments to control the accumulation of the inner solve error in the following sub-section.

\subsection{Inexact Solvers for the Shifted Skew-symmetric System}\label{sec: Inexact Solver for the Shifted Skew Symmetric System}
In practice, we can relax the inner solver to be an inexact approximation. That is we solve equation~\eqref{eq:I+N 2} up to a residual $\varepsilon_{\rm in} = b - \left [ (1+\alpha_k)I +\alpha_k \mu^{-1} \mathcal N \right ] y_{k+1}$. The scheme can be modified to
\begin{subequations}\label{eq:inexact-implicit}
\begin{align}
\label{eq:inexactimex1}    \frac{\hat{x}_{k+1}-x_k}{\alpha_k} &= y_k - \hat{x}_{k+1},  \\
\label{eq:inexactimex2}        \frac{y_{k+1}-y_k}{\alpha_k} &=  \hat{x}_{k+1} - y_{k+1} -\mu^{-1}\left( \nabla F(\hat{x}_{k+1}) + \mathcal Ny_{k+1}\right) - \frac{\varepsilon_{\rm in}}{\alpha_k}, \\
\label{eq:inexactimex3}    \frac{x_{k+1}-x_k}{\alpha_k} &=  y_{k+1} -\frac{1}{2} (x_{k+1} + \hat x_{k+1}).
\end{align}
\end{subequations}
Notice that in the third step~\eqref{eq:inexactimex3}, we use $\frac{1}{2} (x_{k+1} + \hat x_{k+1})$ instead of $x_{k+1}$ for discretization variable $x$ at $k+1$. The perturbation $\varepsilon_{\rm in}$ is the residual of the linear equation \eqref{eq:I+N 2}.

\begin{corollary}\label{cor:inexact IMEX}
Assume function $F \in \mathcal{S}_{\mu, L_F} $. If we compute $y_{k+1}$ such that the residual of~\eqref{eq:inexactimex2}  satisfies
\begin{equation}
\| \varepsilon_{\rm in} \|^2 \leq \frac{\alpha_k}{2 } \left (\| \hat{x}_{k+1} - x_k\|^2 + \alpha_k \| y_{k+1} - \hat x_{k+1}\|^2 \right ),
\end{equation}
then  for $(x_k, y_k)$ be the sequence generated by the inexact accelerated gradient method~\eqref{eq:inexact-implicit} with arbitrary initial value and step size satisfying 
$$
\alpha_k^2L_F \leq (1+\alpha_k/2)\mu,
$$
we have the linear convergence with repect to the Lyapunov function \eqref{eq: acc Lyapunov, mixed}:
\begin{equation*}
    \mathcal{E}_{k+1} \leq \frac{1}{1+\alpha_k/2} \mathcal{E}_k.
\end{equation*}
For $\alpha_k=\sqrt{1/\kappa(F)}$, we achieve the accelerated rate
\begin{equation*}
    \mathcal{E}_k \leq \left( \frac{1}{1+\sqrt{1/\kappa(F)}} \right)^{k} \mathcal{E}_0.
\end{equation*}
\end{corollary}

\begin{proof}
We write~\eqref{eq:inexactimex2} as 
$$
y_{k+1} - y_k = \alpha_k \mathcal G^y (\hat{x}_{k+1}, y_{k+1}) + \varepsilon_{\rm in}.
$$
Compared with Lemma \ref{lem: acc gra method semi-implicit decay}, the inexactness introduces a mis-match term $ \mu(y_{k+1} - x^{\star}, \epsilon_{\rm out})$ in $\langle \partial_y \mathcal E(\hat{x}_{k+1}, y_{k+1}), y_{k+1} - y_k \rangle$ which can be bounded by 
\begin{align*}
 |(y_{k+1} - x^{\star}, \epsilon_{\rm out}) |&\leq \frac{\alpha_k\mu}{4} \|y_{k+1} - x^{\star}\|^2 + \frac{\mu }{\alpha_k }\|\varepsilon_{\rm in}\|^2 \\
&\leq \frac{\alpha_k \mu}{4} \|y_{k+1} - x^{\star}\|^2 + \frac{\mu}{2} \left (\| \hat{x}_{k+1} - x_k\|^2 +  \alpha_k \| y_{k+1} - \hat x_{k+1}\|^2 \right ).
\end{align*}
Use $- \alpha_k \hat{\mathcal{E}}_{k+1} /2$ to cancel the first term and the additional quadratic term in~\eqref{eq:hatdecay} to cancel the second. Then we have 
\begin{equation}\label{eq: implcit-explicit decay}
    \begin{aligned}
    \hat{\mathcal{E}}_{k+1} - \mathcal{E}_k
    \leq - \frac{\alpha_k}{2} \hat{\mathcal{E}}_{k+1} -\alpha_k\left\langle \nabla D_F(\hat x_{k+1}, x^{\star}), y_{k+1}-y_k\right \rangle-\frac{\mu}{2}\left\|y_{k+1}-y_k\right\|^2,
\end{aligned}
\end{equation}
which with the correction step~\eqref{eq:inexactimex3} satisfying assumptions in Lemma \ref{lem:correction step bound} with $c_1= 1/2, c_2 =1$ and $c_3 = \mu$. 
\end{proof}

When we choose $\alpha_k = \sqrt{1/\kappa(F)}$, \eqref{eq:I+N 2} forms a shifted skew-symmetric equation $(\beta I + \mathcal N)y = b(\hat x_{k+1}, y_k)$ with $\beta = (1+ \sqrt{\mu/L_F})\sqrt{\mu L_F}$ . Solvers for the shifted skew-symmetric equation is discussed in Section \ref{sec: gradient methods}. In particular, the inner iteration steps are roughly
$C_{\rm in} \sqrt{\kappa (F)} L_{B^{\sym}}/L_F.$
The outer iteration is $C_{\rm out}\sqrt{\kappa(F)}$. Constants $C_{\rm in} = \mathcal O(|\ln \varepsilon_{\rm in} |)$ and $C_{\rm out} = \mathcal O(|\ln \epsilon_{\rm out}|)$ depend on the tolerance $\varepsilon_{\rm in}$ for the inner iteration and $\epsilon_{\rm out}$ for the outer iteration. Therefore \eqref{eq:inexact-implicit} requires
\begin{itemize}
 \item $\nabla F(x_k)$  gradient evaluation: $C_{\rm out}\sqrt{\kappa(F)}$;
 
 \smallskip
 
 \item $(\beta I + \mathcal N)x$ matrix-vector multiplication: $C_{\rm in} C_{\rm out} \tilde{\kappa}(B^{\sym})$,
\end{itemize}
where we use the relation $\kappa (F) L_{B^{\sym}}/L_F = \tilde{\kappa}(B^{\sym})$.

\subsection{Accelerated Gradient and Skew-symmetric Splitting Methods}
Combining the IMEX scheme in Section \ref{sec: IMEX scheme} and AOR in Section \ref{sec: AOR for nonlinear problems}, we propose the following explicit discretization of the accelerated gradient flow: 
\begin{subequations}\label{eq:explicit}
\begin{align}
\label{eq:AGexplicit1}    \frac{\hat x_{k+1}-x_k}{\alpha} &=  y_k - \hat x_{k+1}, \\
\label{eq:AGexplicit2}   \frac{y_{k+1}-y_k}{\alpha} &= \hat x_{k+1} - y_{k+1}- \frac{1}{\mu}\left( \nabla F(\hat x_{k+1}) + B^{\sym} y_k - 2B y_{k+1} \right), \\
\label{eq:AGexplicit3}      \frac{x_{k+1}-x_{k}}{\alpha} &= y_{k+1} -\frac{1}{2} (x_{k+1} + \hat x_{k+1}).
\end{align}
\end{subequations}
The update of $y_{k+1}$ is equivalent to solve a lower triangular linear algebraic system
\begin{equation*}
  \left( (1+\alpha) I -\frac{2\alpha}{\mu} B \right )y_{k+1} = b(\hat x_{k+1}, y_k)
\end{equation*}
with $b(\hat{x}_{k+1},y_k) = y_k+\alpha \hat x_{k+1} - \frac{\alpha}{\mu} \nabla F(\hat x_{k+1}) - \frac{\alpha}{\mu} (B^{\intercal}+ B) y_k$ which can be computed efficiently by a forward substitution and thus no inexact inner solver is needed. Subtracting~\eqref{eq:AGexplicit3} from~\eqref{eq:AGexplicit1} implies the relation~\eqref{eq:extrapolation} with $c_1=1/2, c_2 = 1$. 

Consider the modified Lyapunov function:
\begin{equation}\label{eq: acc discrete Lyapunov}
\begin{aligned}
\mathcal E^{\alpha B}(x, y) &=  D_F(x, x^{\star}) + \frac{1}{2}\|y-x^{\star}\|^2_{ \mu I -  \alpha B^{\sym}}.
\end{aligned}
\end{equation}
For $\displaystyle 0 <  \alpha <\frac{\mu}{L_{B^{\sym}}}$, we have $\mu I - \alpha B^{\sym}$ is positive definite. Then $\mathcal{E} \geq 0$ and $\mathcal E= 0$  if and only if $x = y = x^{\star}$. We denote 
$$
\mathcal E^{\alpha B}_k =\mathcal E^{\alpha B}(x_k, y_k), \text{ and } \quad \hat{\mathcal{E}}^{\alpha B}_k =\mathcal E^{\alpha B}(\hat{x}_k, y_k).
$$

\begin{lemma}\label{lem: acc gra method decay}
Assume function $F \in \mathcal{S}_{\mu} $. Let $(\hat x_k, y_k)$ be the sequence generated by the accelerated gradient method~\eqref{eq:explicit} with $0< \alpha \leq \displaystyle \frac{\mu}{2L_{B^{\sym}}}$. Then, for $k\geq 0$, the modified Lyapunov function~\eqref{eq: acc discrete Lyapunov} satisfies
\begin{equation}\label{eq:acc gra method decay}
\begin{aligned}
\hat{\mathcal{E}}^{\alpha B}_{k+1} - \mathcal E^{\alpha B}_k \leq&~ -\frac{\alpha}{2} \hat{\mathcal{E}}^{\alpha B}_{k+1} - \alpha\left\langle \nabla D_F(\hat x_{k+1}, x^{\star}), y_{k+1}-y_k\right\rangle -\frac{\mu}{4}\left\|y_{k+1}-y_k\right\|^2. 
\end{aligned}
\end{equation}
\end{lemma}

\begin{proof}
 We write the difference of the modified Lyapunov function into 
 \begin{equation}
    \begin{aligned} \label{eq: E and EB bridge}
\hat{\mathcal{E}}^{\alpha B}_{k+1} - \mathcal E^{\alpha B}_k=&~\hat{\mathcal{E}}_{k+1}  -\mathcal{E}_{k}-\frac{1}{2}\| y_{k+1} - x^{\star}\|^2_{\alpha B^{\sym}}  + \frac{1}{2}\| y_k - x^{\star}\|^2_{\alpha B^{\sym}}. 
\end{aligned}  
\end{equation}
where $\hat{\mathcal{E}}_{k+1}, \mathcal{E}_{k}$ are defined in~\eqref{eq:hatE} and $\mathcal{E}$ refers to the Lyapunov function~\eqref{eq: acc Lyapunov, mixed}.
Since $\mathcal E$ is $\mu$-convex,
 \begin{align*}
\hat{\mathcal{E}}_{k+1}  -\mathcal{E}_{k}  \leq{}& \langle \partial_x \mathcal{E}(\hat x_{k+1}, y_{k+1}) , \hat{x}_{k+1} - x_k \rangle -\frac{\mu }{2} \| \hat{x}_{k+1} - x_k\|^2 \\
&~ \langle \partial_y  \mathcal{E}(\hat{x}_{k+1}, y_{k+1}), y_{k+1} - y_k \rangle -\frac{\mu }{2} \| y_{k+1} - y_k\|^2. 
\end{align*}
Then write the scheme as a correction of the implicit Euler scheme 
\begin{align*}
\hat{x}_{k+1} - x_k &= \alpha \mathcal G^x (\hat{x}_{k+1}, y_{k+1}) +  \alpha (y_k - y_{k+1}) ,\\
y_{k+1} - y_k & = \alpha \mathcal G^y (\hat{x}_{k+1}, y_{k+1}) +\frac{\alpha}{\mu} B^{\sym}  (y_{k+1} - y_{k}).
\end{align*}

According to the proof of Lemma \ref{lem: acc gra method semi-implicit decay}, we get 
\begin{equation}\label{eq:intermediate decay}
\begin{aligned}
\hat{\mathcal{E}}_{k+1} - \mathcal{E}_k \leq&~ -\alpha \hat{\mathcal{E}}_{k+1}-\alpha\left\langle \nabla D_F(\hat x_{k+1}, x^{\star}), y_{k+1}-y_k\right\rangle -\frac{\mu}{2}\left\|y_{k+1}-y_k\right\|^2  \\
&+ (y_{k+1} - x^{\star}, y_{k+1} - y_k )_{\alpha B^{\sym}}.
\end{aligned}
\end{equation}
We use the identity~\eqref{eq:squares} to expand the last cross term in \eqref{eq:intermediate decay} as
\begin{align*}
	&(y_{k+1} - x^{\star}, y_{k+1} - y_k )_{\alpha B^{\sym}} \\
	&=  \frac{1}{2}\| y_{k+1} -x^{\star} \|_{\alpha  B^{\sym}}^2 + \frac{1}{2} \| y_{k+1} -y_k \|_{ \alpha B^{\sym}}^2 - \frac{1}{2} \| y_{k} -x^{\star} \|_{\alpha  B^{\sym}}^2.
\end{align*}
Substitute back to~\eqref{eq:intermediate decay} and rearrange terms using~\eqref{eq: E and EB bridge}, 
\begin{equation*}
\begin{aligned}
\hat{\mathcal{E}}^{\alpha B}_{k+1} - \mathcal E^{\alpha B}_k \leq&~ -\alpha \hat{\mathcal{E}}_{k+1}-\alpha\left\langle \nabla D_F(\hat x_{k+1}, x^{\star}), y_{k+1}-y_k\right\rangle -\frac{1}{2}\left\|y_{k+1}-y_k\right\|^2_{\mu I -  \alpha B^{\sym}} \\
=&~ -\frac{\alpha}{2} \hat{\mathcal{E}}^{\alpha B}_{k+1} -\frac{\alpha}{2} D_F(\hat x_{k+1}, x^{\star})- \frac{\alpha }{4} \|y_{k+1} - x^{\star}\|_{\mu I + \alpha B^{\sym}} \\
&-\alpha\left\langle \nabla D_F(\hat x_{k+1}, x^{\star}), y_{k+1}-y_k\right\rangle 
-\frac{1}{2}\left\|y_{k+1}-y_k\right\|^2_{ \mu I-\alpha B^{\sym}} \\
=&~ -\frac{\alpha}{2} \hat{\mathcal{E}}^{\alpha B}_{k+1}-\alpha\left\langle \nabla D_F(\hat x_{k+1}, x^{\star}), y_{k+1}-y_k\right\rangle -\frac{\mu}{4}\left\|y_{k+1}-y_k\right\|^2\\
- \frac{\alpha}{2} D_F&(\hat x_{k+1}, x^{\star})- \frac{\alpha}{4} \|y_{k+1} - x^{\star}\|_{\mu I + \alpha B^{\sym}}-\frac{1}{4}\left\|y_{k+1}-y_k\right\|^2_{\mu I- 2\alpha B^{\sym}} .
\end{aligned}
\end{equation*}
According to our choice $\alpha \leq \displaystyle \frac{\mu}{2L_{B^{\sym}}}$,  both $\mu I - 2\alpha B^{\sym}$ and $\mu I + \alpha B^{\sym}$ are SPD. Dropping the last three non-positive terms we have the desired result.
\end{proof}

The decay of the modified Lyapunov function~\eqref{eq:acc gra method decay} and the appropriate relation~\eqref{eq:extrapolation} satisfies assumptions of Lemma \ref{lem:correction step bound} with $c_1=1/2, c_2 = 1$, and $c_3 = \mu/2$. Although the Lyapunov function is slightly different,~\eqref{eq: correction decay} still holds for $\mathcal E^{\alpha B}$. We conclude with the following linear convergence rate.

\begin{theorem}\label{thm: linear convergence of AGSS explicit scheme}
Assume function $F \in \mathcal{S}_{\mu, L_F} $. Let $(x_k, y_k)$ be the sequence generated by the accelerated gradient method~\eqref{eq:explicit} with arbitrary initial value and step size satisfying $$0< \alpha \leq \min \left \{\frac{\mu}{2L_{B^{\sym}}}, \sqrt{\frac{ \mu}{2L_F}}\right \},$$ then for the modified Lyapunov function~\eqref{eq: acc Lyapunov, mixed},
\begin{equation*}
    \mathcal E^{\alpha B}_{k+1} \leq \frac{1}{1+\alpha/2 } \mathcal E^{\alpha B}_k.
\end{equation*}
In particular, $\displaystyle \alpha = \min \left \{\frac{\mu}{2L_{B^{\sym}}}, \sqrt{\frac{ \mu}{2L_F}}\right \}$, we achieve the accelerated rate
\begin{equation*}
   \|x_{k+1}-x^{\star}\|^2 \leq \frac{2}{\mu} \mathcal E^{\alpha B}_k \leq \left (1+ 1/\max\left \{ 4 \tilde{\kappa}( B^{\sym}), \sqrt{8\kappa(F)}\right \}\right)^{-k} \frac{2\mathcal{E}_0^{\alpha B}}{\mu}.
\end{equation*}
\end{theorem}




As expected, we have developed an explicit scheme~\eqref{eq:explicit} which achieves the accelerated rate.
Therefore the cost of gradient $ \nabla F(x_k)$ evaluation and matrix-vector $(\beta I + B)^{-1}b$ multiplication are both $C_{\rm out}\max\{\tilde{\kappa}( B^{\sym}), \sqrt{\kappa(F)}\}$. Compare with the inexact IMEX scheme \eqref{eq:semi-implicit}, we may need more gradient evaluation but less matrix-vector multiplication. If $\tilde{\kappa}( B^{\sym}) \gg \sqrt{\kappa(F)}$ and $\nabla F(x_k)$ is computationally expensive to evaluate, then the IMEX scheme \eqref{eq:semi-implicit} or its inexact version \eqref{eq:inexact-implicit} is favorable; otherwise if the error tolerance for the inexact inner solve is small, i.e., $C_{\rm in} $ is large, and the cost of $(\beta I + \mathcal N)x$ is comparable to the evaluation of $ \nabla F(x_k)$, then the explicit scheme \eqref{eq:explicit} takes advantage.

\section{Nonlinear Saddle Point Systems}\label{sec:saddle point systems}
In this section, we derive algorithms for strongly-convex-strongly-concave saddle point system which achieve optimal lower complexity. Moreover, by taking advantage of the block structure of the saddle point problems, we introduce the preconditioned version of GSS and AGSS methods.

\subsection{Problem Setting}
Consider the nonlinear smooth saddle point system with bilinear coupling:
\begin{equation}\label{eq: min-max problem}
    \min_{u\in \mathcal V} \max_{p \in \mathcal Q} \mathcal{L}(u,p) = f(u) - g(p) + (Bu - b,p)
\end{equation}
where $\mathcal V=\mathbb{R}^m, \mathcal Q=\mathbb{R}^n, m\geq n$ are finite-dimensional Hilbert spaces with inner product induced by SPD operators $\IV, \IQ$, respectively. Functions $f(u) \in \mathcal S_{\mu_{f}, L_{f} },$ and $g(p) \in \mathcal S_{\mu_{g}, L_{g} }$, are measured in $\IV$ norm and $\IQ$ norm, respectively. The operator $B$ is an $n\times m$ matrix of full rank. 

The point $(u^{\star}, p^{\star})$ solves the min-max problem~\eqref{eq: min-max problem} is said to be a saddle point of $\mathcal{L}(u,p)$, that is
$$\mathcal{L}(u^{\star}, p) \leq \mathcal{L}(u^{\star}, p^{\star}) \leq  \mathcal{L}(u, p^{\star})\quad \forall \ (u,p)\in \mathbb R^m \times \mathbb R^n.$$ 
The saddle point $(u^*, p^*)$ satisfies the first order necessary condition for being the critical point of $\mathcal{L}(u, p)$:
\begin{subequations}\label{eq:critical point system}
\begin{align}
\label{eq:du}    \nabla f(u^{\star}) + B^{\intercal}p^{\star} = 0, \\
\label{eq:dp}    -Bu^{\star} + \nabla g(p^{\star}) = 0.
\end{align}
\end{subequations}
The first equation~\eqref{eq:du} is $\partial_u \mathcal L(u^{\star}, p^{\star})= 0$ but the second one~\eqref{eq:dp} is $- \partial_p \mathcal L(u^{\star}, p^{\star})= 0$. The negative sign is introduced so that  ~\eqref{eq:critical point system} is in the form  $$\mathcal A(x^{\star}) = 0,$$
where $x = (u,p) \in \mathcal V \times \mathcal Q$, $\nabla F = 
\begin{pmatrix}
 \nabla f & 0\\
 0 & \nabla g
\end{pmatrix}
$, and $\mathcal N = \begin{pmatrix}
0 & \;  B^{\intercal} \\
-B & \;  0
\end{pmatrix}$. To avoid confusion, we use the notation $\mathcal B^{\sym} = \begin{pmatrix}
0 & \;  B^{\intercal} \\
B & \;  0
\end{pmatrix} $ in this section. The splitting ~\eqref{eq:Bs-B} becomes
$$\mathcal N 
=  \begin{pmatrix}
0 & \;  B^{\intercal} \\
B & \;  0
\end{pmatrix} -  \begin{pmatrix}
0 & \; 0 \\
2B & \; 0
\end{pmatrix},
\quad 
\text{ and }
\quad
\mathcal N = 
\begin{pmatrix}
 I & \; 0 \\
 0 & \; -I
\end{pmatrix}
\mathcal B^{\sym}.
$$
Therefore $\|\mathcal N\| = \| \mathcal B^{\sym} \|$ for any operator norm.

Given two SPD operators $\IV$ and $\IQ$, we denote by $\IX :=
\begin{pmatrix}
\IV & 0\\
0 & \IQ
\end{pmatrix},$
and
$\mu \IX :=
\begin{pmatrix}
\mu_f\IV & 0\\
0 & \mu_g\IQ
\end{pmatrix}.$
Then for any $ x = (u,p),y = (v,q) \in  \mathcal V \times \mathcal Q,$
\begin{equation*}
    (\mathcal{A}(x) - \mathcal{A}(y), x-y) \geq  \|x-y\|_{\mu \IX}^2.
\end{equation*}
The accelerated gradient flow and the discrete schemes follows from discussion in Section \ref{sec:AGSS flow and schemes}. The results are slightly sharp by including the preconditioner $\IX$ as a block diagonal matrix.

\subsection{Accelerated Gradient Flow for Saddle Point Problems}
The component form of \eqref{eq:AG} becomes the preconditioned accelerated gradient flow for the saddle point system:
\begin{equation}\label{eq:AG saddle}
\begin{aligned}
     u^{\prime} &= v - u ,\\
     p^{\prime} &= q - p,\\
     v^{\prime} & = u - v - \mu_f^{-1}\IV^{-1}(\nabla f(u) + B^{\intercal}q), \\
     q^{\prime} & =p - q -\mu_g^{-1}\IQ^{-1}(\nabla g(p) - Bv).
\end{aligned}
\end{equation}
Recall that $\IV^{-1}, \IQ^{-1}$ are SPD operators considered as preconditioner or change of metric. We employ the inverse notation based on the intuition that a preconditioner can serve as a linear solver to approximate the matrix inverse, and $\IV, \IQ$ will become a novel metric for the variable spaces. Our analysis is designed to provide a general framework that offers the flexibility to incorporate various metrics tailored to specific problems. In particular, when $\IV^{-1}, \IQ^{-1}$ are identities, then \eqref{eq:AG saddle} reduced to the accelerated gradient flow \eqref{eq:AG} under typical Euclidean metric. 
 
Consider the Lyapunov function:
\begin{equation}\label{eq: acc Lyapunov saddle}
\mathcal{E}(x, y) := D_F(x,x^{\star}) +\frac{1}{2}\| y - x^{\star}\|_{\mu \IX}^2,
\end{equation}
where $ D_F(x,x^{\star}) := D_f(u, u^{\star}) +  D_g(p, p^{\star})$. As $f, g$ are strongly convex, $\mathcal{E}(x, y)\geq 0$ and $\mathcal{E}(x, y)= 0$ iff $x=y=x^{\star}$. 
%
Denote the vector field on the right hand side of~\eqref{eq:AG saddle} by $\mathcal G(x,y)$. It is straightforward to verify the strong Lyapunov property following Lemma \ref{lem: acc strong Lyapunov}:
\begin{equation}\label{eq:strong saddle}
\begin{aligned}
- \nabla \mathcal E(x,y)\cdot \mathcal G(x,y) \geq \mathcal E(x,y)+\frac{1}{2}\|y-x\|^2_{\mu \IX}.
\end{aligned}
\end{equation}

\subsection{GSS and AGSS Methods for Saddle Point Problems}

\subsubsection{GSS method}

For saddle point problems, the preconditioned GSS method \eqref{eq: nonlinear SOR 1} is written as 
\begin{equation}\label{eq:GSS_saddle}
\begin{aligned}
     \frac{u_{k+1} - u_k}{\alpha} &= -\mu_f^{-1} \IV^{-1} \left (\nabla f(x_k)+ B^{\intercal}p_k \right), \\
      \frac{p_{k+1} - p_k}{\alpha} &= -\mu_g^{-1}\IQ^{-1}\left (\nabla g(x_k)+ Bu_k - 2Bu_{k+1}\right).
\end{aligned}
\end{equation}
Consider the modified  Lyapunov function
 \begin{equation}\label{eq:Lyapunov_GSS_saddle}
    \mathcal E^{\alpha BD}(x_k):= \frac{1}{2}\|x_k-x^{\star}\|^2_{\mu \IX - \alpha \mathcal B^{\sym}}  -\alpha D_F(x^{\star}, x_k)
\end{equation}

 We first show $\mathcal E^{\alpha BD}$ is non-negative if $\alpha$ is sufficiently small. 

\begin{lemma}\label{lem: positivity of Lyapunov}
Assume $f\in \mathcal{S}_{\mu_{f}}$ and  $g\in \mathcal{S}_{\mu_{g}}$. When the step size $\alpha \leq  \sqrt{\mu_f\mu_g/(4L_S)}$ with $L_{S} = \lambda_{\max} (\IQ^{-1} B\IV^{-1}B^{\intercal})$, the matrix
$$
\mu \IX - 2\alpha \mathcal B^{\sym} =
\begin{pmatrix}  \mu_f\IV & -2\alpha B^{\intercal} \\
    -2\alpha B &  \mu_g\IQ
   \end{pmatrix}
$$ 
is symmetric and positive semidefinite.
%
\end{lemma}
\begin{proof}
We have the block matrix factorization
$$
\begin{aligned}
& \begin{pmatrix} \mu_f\IV & -2\alpha B^{\intercal} \\
    -2\alpha B &  \mu_g\IQ
   \end{pmatrix} \\
   &=\begin{pmatrix} I_m & 0 \\
     -2\alpha \mu_f^{-1}B\IV^{-1} & I_n
   \end{pmatrix}\begin{pmatrix}  \IV & 0 \\
    0 & \mu_g\IQ- 4\alpha^2 \mu_f^{-1}B\IV^{-1}B^{\intercal}
   \end{pmatrix} \begin{pmatrix}I_m & -2\alpha \mu_f^{-1} \IV^{-1} B^{\intercal} \\
  0  & I_n
   \end{pmatrix}.
\end{aligned}
$$
Then if $\alpha \leq  \sqrt{\mu_f\mu_g/(4L_S)}$, we have $\mu_g\IQ- 4\alpha^2 \mu_f^{-1}B\IV^{-1}B^{\intercal}\geq 0$ and the results follows.
\end{proof}

As a result, one can easily show that $\mathcal{E}^{\alpha B} \geq 0$ if $$\alpha < \min \{\sqrt{\mu_f\mu_g/(4L_S)}, \mu_g/(2L_f), \mu_g/(2L_g)\}$$ and  $\mathcal E^{\alpha BD}(x)= 0$  if and only if $x = x^{\star}$. We skip the proof since it is similar to that of Lemma \ref{lem:positivity of GSS Lyapunov}.

The convergence analysis will depend on the condition number of the Schur complement $S = B\IV^{-1}B^{\intercal}$ in the $\IQ$ inner product.

\begin{lemma}\label{lem: LB saddle}
  For $\mathcal B^{\sym} = \begin{pmatrix}
        0 & \;  B^{\intercal} \\
        B & \;  0
    \end{pmatrix}$, we have 
    \begin{equation*}
        L_{\mathcal B^{\sym}} := \|\mathcal B^{\sym}\|_{\IX} =\sqrt{ L_S} = \lambda_{\max}^{1/2}(\IQ^{-1} B\IV^{-1}B^{\intercal}).
    \end{equation*}
\end{lemma}

\begin{proof}
    Since $\mathcal B^{\sym} \in \mathbb{R}^{(m+n)^2}$ is symmetric, $\mathcal B^{\sym}$ has $(m+n)$ real eigenvalues. Let $\lambda$ be the largest absolute value eigenvalue with respect to $\IX$ and $x = (u,p)$ be the corresponding eigenvector, that is $ \|\mathcal B^{\sym}\|_{\IX} = \lambda$ and 
    \begin{equation*}
        \mathcal B^{\sym} x = \lambda \IX x.
    \end{equation*}
    The component form follows as 
    \begin{equation*}
        \begin{aligned}
             Bu= \lambda \IQ p, \quad  B^{\intercal}p = \lambda \IV u.
        \end{aligned}
    \end{equation*}
    Substitute $u$ in the first equation side using the other equation, we get
    \begin{equation*}
                \begin{aligned}
            B\IV^{-1}B^{\intercal}p = \lambda^2\IQ p.
        \end{aligned}
    \end{equation*}
    Since $B\IV^{-1}B^{\intercal}$ is positive definite, we have $|\lambda|  = \sqrt{L_S}$, $L_S = \lambda_{\max}(\IQ^{-1}B\IV^{-1}B^{\intercal})$.
\end{proof}


Define $\mathcal{E}_k^{\alpha BD} := \mathcal E^{\alpha BD}(x_k)$ and $\tilde{\kappa}_{\IX}(\mathcal N):=\sqrt{L_S/(\mu_f\mu_g)}$. With $\kappa(F)$ refined to $\kappa_{\IV}(f)$ and $\kappa_{\IQ}(g)$, 
we state the linear convergence of the GSS method for saddle point problems as a direct corollary of Theorem \ref{thm: linear convergence for generalized nonlinear system}.

\begin{theorem}\label{thm: linear convergence for GSS_saddle}
Let $\{u_k, p_k\}$ be the sequence generated by the GSS method~\eqref{eq:GSS_saddle} with arbitrary initial guess $(u_0, p_0)$ and step size $$0<  \alpha< \displaystyle \min \left \{ \frac{1}{2 \tilde{\kappa}_{\IX}(\mathcal N)}, \frac{1}{2\kappa_{\IV}(f)}, \frac{1}{2\kappa_{\IQ}(g)} \right \}.$$ 
Then for the discrete Lyapunov function~\eqref{eq:Lyapunov_GSS_saddle}, we have 
\begin{equation*}
\begin{aligned}
\mathcal{E}^{\alpha BD}(x_{k+1})\leq&~ \frac{1}{1+ \alpha}\mathcal{E}^{\alpha BD}(x_k).
\end{aligned}
\end{equation*}
In particular, for $ \alpha = \displaystyle 1/ \max \left \{ 4 \tilde{\kappa}_{\IX}(\mathcal N), 4\kappa_{\IV}(f), 4\kappa_{\IQ}(g) \right \}$, we have
\begin{equation*}
\| x_k - x^{\star}\|_{\mu \IX}^2 \leq \left (1 + 1/ \max \left \{ 4 \tilde{\kappa}_{\IX}(\mathcal N), 4\kappa_{\IV}(f), 4\kappa_{\IQ}(g) \right \}\right )^{-k} 6\| x_0 - x^{\star}\|_{\mu \IX}^2.
\end{equation*}
\end{theorem}

\subsubsection{Implicit in the gradient part}
We can use the generalized gradient flow \eqref{eq: gradient flow} and treat $\nabla F$ implicitly and $\mathcal N$ explicitly with AOR
\begin{equation}\label{eq:IMEX saddle prox}
\begin{aligned}
    u_{k+1} &= \operatorname{prox}_{\alpha/\mu_f f, \IV}(u_k - \alpha \mu_f^{-1} \IV^{-1}B^{\intercal}p_k), \\
   p_{k+1}  &= \operatorname{prox}_{\alpha/\mu_g g, \IQ}(p_k -  \alpha \mu_g^{-1} \IQ^{-1}B(u_k - 2u_{k+1})).
\end{aligned}
\end{equation}
%
When $\IV, \IQ$ are identities, that matched the typical proximal operators defined in~\eqref{eq:proxl2}. The parameters $\mu_f$ and $\mu_g$ are proper scaling for later analysis. Notice that the proximal operation works for non-smooth functions. 

Consider the Lyapunov function for AOR methods: 
\begin{equation}\label{eq: Lyapunov saddle prox}
\begin{aligned}
    	\mathcal{E}^{\alpha B}(u, p) &= \frac{\mu_f}{2}\|u-u^{\star}\|_{\IV}^2 + \frac{\mu_g}{2}\|p-p^{\star}\|_{\IQ}^2 - \alpha  (B(u-u^{\star}), p- p^{\star})\\
     &= \frac{1}{2} \| x- x^{\star} \|^2_{ \mu \IX - \alpha \mathcal B^{\sym}}.
\end{aligned}
\end{equation}
Similar to Lemma \ref{lem: positivity of Lyapunov}, on can show $\mathcal{E}^{\alpha B}(x) \geq 0$ for $0\leq \alpha<1/\tilde{\kappa}_{\IX}(\mathcal N)$ and $\mathcal{E}^{\alpha B}(x)= 0$  if and only if $x = x^{\star}$. As $\nabla F$ is treated implicitly, the following result can be proved similar to Theorem \ref{thm: linear convergence for the shifted skew-symmetric linear system}. Recall that $$\tilde{\kappa}_{\IX}(\mathcal N):=\sqrt{L_S/(\mu_f\mu_g)} =  \sqrt{ \lambda_{\max}(\IQ^{-1}B\IV^{-1}B^{\intercal})/(\mu_f\mu_g)}.$$ 

\begin{theorem}\label{thm: linear convergence for the saddle prox}
Assume $f\in \mathcal{S}_{\mu_{f}}$ and  $g\in \mathcal{S}_{\mu_{g}}$. Let $\{x_k\}$ be the sequence generated by ~\eqref{eq:IMEX saddle prox} with arbitrary initial guess $x_0$ and step size $0< \alpha< 1/\tilde{\kappa}_{\IX}(\mathcal N)$. Then for the Lyapunov function~\eqref{eq: Lyapunov saddle prox},
	\begin{equation}\label{eq:prox saddle rate}
		\begin{aligned}
			\mathcal{E}^{\alpha B}( x_{k+1}) \leq&~ \frac{1}{1+\alpha }\mathcal{E}^{\alpha B} (x_{k}).
		\end{aligned}
	\end{equation}
In particular, for $ \alpha = 1/(2\tilde{\kappa}_{\IX}(\mathcal N))$, we have
\begin{equation*}
\| x_k - x^{\star}\|^2_{ \mu \IX} \leq \left (1 + 1/(2\tilde{\kappa}_{\IX}(\mathcal N)\right )^{-k} 3\| x_0 - x^{\star}\|^2_{  \mu \IX}.
\end{equation*}
\end{theorem}
As a special case when $\IV$ and $\IQ$ are identity matrices,  $\tilde{\kappa}_{\IX}(\mathcal N) = \|B\|/\sqrt{\mu_f\mu_g}$ is indeed the optimal lower bound shown in~\cite{zhang2022lower}. Proximal methods combining 
overrelaxation parameters are considered and yielding optimal rates; see~\cite{chambolle2011first,chambolle2016ergodic,jacobs2019solving} for variants of schemes and applications. Here we extend the framework to proximal operators under general metric other than Euclidean norm, which has wider application for example in~\cite{el2017general,hou2013linear}. 

\subsubsection{AGSS method}

Recall that $x = (u,p),y = (v,q)$. The accelerated gradient and skew-symmetric splitting method is:
\begin{subequations}\label{eq:explicit saddle}
\begin{align}
 \frac{\hat x_{k+1}-x_k}{\alpha} &= y_k - \hat x_{k+1}, \\
\label{eq:explicitsaddle1}       \frac{v_{k+1}-v_k}{\alpha} &=\hat{u}_{k+1} - v_{k+1} -\mu_f^{-1}\IV^{-1} \left( \nabla f(\hat{u}_{k+1}) + B^{\intercal}q_{k}\right) , \\
\label{eq:explicitsaddle2}        \frac{q_{k+1}-q_k}{\alpha} &=  \hat{p}_{k+1} - q_{k+1}-\mu_g^{-1}\IQ^{-1}\left( \nabla g(\hat{p}_{k+1}) -2Bv_{k+1} + Bv_k \right) , \\
   \frac{x_{k+1}-x_{k}}{\alpha} &=y_{k+1} - x_{k+1} + \frac{1}{2}(x_{k+1} - \hat x_{k+1}).
\end{align}
\end{subequations}
Each iteration requires $2$ matrix-vector products if we store $Bv_k$, $2$ gradient evaluations and the computation of $\IV^{-1}$ and $\IQ^{-1}$.
Notice that the scheme is explicit as $v_{k+1}$  can be first updated by~\eqref{eq:explicitsaddle1} and then used to generate $q_{k+1}$ in~\eqref{eq:explicitsaddle2}.

Consider the tailored discrete Lyapunov function:
\begin{equation}\label{eq: acc discrete Lyapunov saddle}
\begin{aligned}
\mathcal E^{\alpha B}(x, y) :={}& D_F(x,x^{\star}) + \frac{1}{2}\| y - x^{\star}\|_{\mu \IX - \alpha \mathcal B^{\sym}}^2. 
\end{aligned}
\end{equation}
As a result of Lemma \ref{lem: positivity of Lyapunov}, $\mathcal{E}^{\alpha B} \geq 0$ if $\alpha \leq  \sqrt{\mu_f\mu_g/(4L_S)}$ and  $\mathcal{E}^{\alpha B}(x, y) = 0$ only if $x = x^{\star}$. Using the first order necessary conditions ~\eqref{eq:critical point system}, the Bregman divergence part can be related to the duality gap $\Delta (u,p):= \max_q \mathcal L(u, q) - \min_v \mathcal L(v, p)  $:
\begin{equation*}
    \begin{aligned}
    D_F(x, x^{\star})= &~  D_f(u; u^{\star}) + D_g(p; p^{\star}) \\
    =&~  f(u) - f(u^{\star}) - \langle \nabla f(u^{\star}), u - u^{\star}\rangle + g(p) -g(p^{\star}) - \langle  \nabla g(p^{\star}), p - p^{\star}\rangle \\
    =&~f(u) - f(u^{\star}) + \langle B^{\intercal}p^{\star}, u - u^{\star}\rangle + g(p) -g(p^{\star}) - \langle Bu^{\star}, p - p^{\star}\rangle \\
    =&~ \mathcal L(u, p^{\star}) - \mathcal L(u^{\star}, p) \leq \Delta (u,p) . 
    \end{aligned}
\end{equation*}

Define $\mathcal{E}_k^{\alpha B} := \mathcal E^{\alpha B}(x_k, y_k)$ and recall that $\tilde{\kappa}_{\IX}(\mathcal N)=\sqrt{L_S/(\mu_f\mu_g)}$. With $\kappa(F)$ refined to $\kappa_{\IV}(f)$ and $\kappa_{\IQ}(g)$, 
we have the following linear convergence as a direct corollary of Theorem \ref{thm: linear convergence of AGSS explicit scheme}.

\begin{theorem}\label{thm: linear convergence of explicit scheme saddle}
Assume $f\in \mathcal{S}_{\mu_{f}, L_{f}}$ and  $g\in \mathcal{S}_{\mu_{g}, L_{g}}$. Let $(x_k, y_k)$ be the sequence generated by the accelerated gradient and skew-symmetric splitting method~\eqref{eq:explicit saddle}  with arbitrary initial value and step size satisfying 
$$
0<  \alpha \leq \displaystyle \min \left \{ \frac{1}{2 \tilde{\kappa}_{\IX}(\mathcal N)}, \sqrt{\frac{1}{2\kappa_{\IV}(f)}}, \sqrt{\frac{1}{2\kappa_{\IQ}(g)}} \right\}.
$$
Then for the discrete Lyapunov function~\eqref{eq: acc discrete Lyapunov saddle},
\begin{equation*}
    \mathcal{E}_{k+1}^{\alpha B}\leq \frac{1}{1+ \alpha/2} \mathcal{E}_{k}^{\alpha B}.
\end{equation*}
In particular, for $\alpha =  1/\max \left \{2\tilde{\kappa}_{\IX}(\mathcal N), \sqrt{2\kappa_{\IV}(f)}, \sqrt{2\kappa_{\IQ}(g)}\right \}$, we achieve the accelerated rate
\begin{equation*}
\begin{aligned}
&\mu_f \|u_{k}-u^{\star}\|_{\IV}^2 + \mu_g\|p_{k}-p^{\star}\|_{\IQ}^2  \\
&\leq \left (1+ 1/\max \left \{4\tilde{\kappa}_{\IX}(\mathcal N), 2\sqrt{2\kappa_{\IV}(f)}, 2\sqrt{2\kappa_{\IQ}(g)}\right \}\right)^{-k} 2\mathcal{E}_0^{\alpha B}.
\end{aligned}
\end{equation*}
\end{theorem}

For strongly-convex-strongly-concave saddle point systems, set $\IV$ and $\IQ$ to identity matrices, scheme~\eqref{eq:explicit saddle} is an explicit scheme achieving the lower complexity bound: $\Omega\left(\sqrt{\kappa(f)+\kappa^2(\mathcal N)+\kappa(g)} |\ln \epsilon |\right)$  
established in ~\cite{zhang2022lower}. Notice that in~\cite{zhang2022lower}, only the theoretical lower bound is proved and no algorithms are developed to match the lower bound. Ours are among the first few explicit schemes~\cite{jin2022sharper,kovalev2022accelerated,thekumparampil2022lifted} achieving this lower bound. The optimality of our algorithm is confirmed through the solution of empirical risk minimization in Section \ref{sec: Empirical risk minimization}.

The condition numbers can be improved using appropriate SPD preconditioners $\IV$ and $\IQ$ with the price of computing $\IV^{-1}$ and $\IQ^{-1}$. The term $\tilde{\kappa}_{\IX}(\mathcal N)$ might be the leading term compare with $\sqrt{\kappa_{\IV}(f)}$ and $\sqrt{\kappa_{\IQ}(g)}$. 
Preconditioner $\IQ^{-1}$ can be chosen so that $L_S$ is small. Namely $\IQ^{-1}$ is a preconditioner for the Schur complement $B\IV^{-1}B^{\intercal}$. For example, if $\IV = I_m$, then the ideal choice is $\IQ = (BB^{\intercal})^{-1}$ so that $L_S = 1$. Such choice of $\IQ$ may increase the condition number $\kappa_{\IQ}(g)$ as $(BB^{\intercal})^{-1}$ may not be a good preconditioner of $g$. Transformed primal dual methods developed in~\cite{chen2023transformed} will provide a better approach but cannot fit into our approach.

\subsubsection{Implicit in the skew-symmetric part}\label{sec: IMEX AGSS saddle}
The term $\tilde{\kappa}_{\IX}(\mathcal N)$ is due to the bilinear coupling or equivalently from the explicit treatment for the skew-symmetric component $\mathcal N$. We can treat $\mathcal N$ implicitly and obtain the following IMEX scheme:
\begin{subequations}\label{eq:semi-implicit saddle}
\begin{align}
   \frac{\hat{x}_{k+1}-x_k}{\alpha_k} &= y_k - \hat{x}_{k+1},  \\
       \frac{v_{k+1}-v_k}{\alpha_k} &= \hat{u}_{k+1} - v_{k+1}-  \label{eq:IMEXsaddlev} \mu_f^{-1}\IV^{-1} \left( \nabla f(\hat{u}_{k+1}) + B^{\intercal}q_{k+1}\right) , \\
\label{eq:IMEXsaddleq}      \frac{q_{k+1}-q_k}{\alpha_k} &= \hat{p}_{k+1} - q_{k+1} - \mu_g^{-1}\IQ^{-1}\left( \nabla g(\hat{p}_{k+1}) -Bv_{k+1}\right) , \\
   \frac{x_{k+1}-x_k}{\alpha_k} &= y_{k+1} - x_{k+1}.
\end{align}
\end{subequations}
As the skew-symmetric part is treat implicitly, the restriction $\alpha\leq 1/(2\tilde{\kappa}_{\IX}(\mathcal N))$ can be removed. We state the convergence result directly as the proofs are similar to discussion in Section \ref{sec: IMEX scheme} for IMEX schemes.

\begin{theorem}\label{thm: linear convergence of saddle IMEX scheme}
Assume $f\in \mathcal{S}_{\mu_{f}, L_{f}}$ and  $g\in \mathcal{S}_{\mu_{g}, L_{g}}$. Let $(x_k, y_k)$ be the sequence generated by the preconditioned accelerated gradient method~\eqref{eq:semi-implicit saddle} with arbitrary initial value and step size $\alpha_k$ satisfying $$ \alpha_k^2 L_{f}  \leq ( 1+\alpha_k) \mu_{f} , \quad \alpha_k^2 L_{g} \leq (1+\alpha_k) \mu_{g} ,$$ 
then for the Lyapunov function~\eqref{eq: acc Lyapunov saddle},
\begin{equation}\label{eq:IMEX decay saddle}
    \mathcal{E}_{k+1} \leq   \frac{1}{1+\alpha_k} \mathcal{E}_k.
\end{equation}
In particular for $ \alpha_k = 1/ \max \{\sqrt{\kappa_{\IV}(f)}, \sqrt{\kappa_{\IQ}(g)}\}$, we achieve the accelerated rate
\begin{equation*}
    \mathcal{E}_k \leq \left( 1+1/ \max \{\sqrt{\kappa_{\IV}(f)}, \sqrt{\kappa_{\IQ}(g)}\} \right)^{-k} \mathcal{E}_0.
\end{equation*}
\end{theorem}


We discuss the inner solve on~\eqref{eq:IMEXsaddlev}-\eqref{eq:IMEXsaddleq} which is equivalent to the following linear algebraic equation
\begin{equation}\label{eq:I+N saddle}
\begin{pmatrix}
 (1+\alpha_k)\mu_f \IV & \alpha_k B^{\intercal} \\
 -\alpha_k B & (1+\alpha_k) \mu_g \IQ
\end{pmatrix}\begin{pmatrix}
 v_{k+1} \\
 q_{k+1}
\end{pmatrix}  = b(\hat{x}_{k+1},y_k),
\end{equation}
with $b(\hat{x}_{k+1},y_k) = \IX( y_k + \alpha_k \hat x_{k+1} )- \alpha_k \nabla F(\hat x_{k+1})$. 
Solving equation~\eqref{eq:I+N saddle} basically costs the effort of solving
$$((1+\alpha_k)^2 \mu_f \mu_g \IQ + \alpha_k^2 B\IV^{-1}B^{\intercal})x = b,$$
which can be solved by preconditioned conjugate gradient methods.

As a direct application of Section \ref{sec: Inexact Solver for the Shifted Skew Symmetric System}, the inner solver can be inexact. We denote the tolerance by $|\ln \epsilon_{\rm in}|$. With the restriction on the step size $\alpha_k$, the outer iteration complexity is $O(\alpha_k^{-1} |\ln \epsilon_{\rm out}|)$.  
Therefore we conclude that IMEX scheme~\eqref{eq:semi-implicit saddle} requires
\begin{itemize}
\item gradient evaluation:  $\max\{\sqrt{\kappa_{\IV} (f)}, \sqrt{\kappa_{\IQ}(g)}\} |\ln \epsilon_{\rm out}|$,
\smallskip
\item matrix-vector multiplication: $ \tilde{\kappa}_{\IX}(\mathcal N) |\ln \epsilon_{\rm out}| |\ln \epsilon_{\rm in}|$,
\smallskip
\item preconditioners $\IV^{-1}$ and $\IQ^{-1}$: $ \displaystyle \tilde{\kappa}_{\IX}(\mathcal N)|\ln \epsilon_{\rm out}| |\ln \epsilon_{\rm in}|$,
\end{itemize}
which matches the optimal lower bound in~\cite{zhang2022lower}. 
The preconditioners $\IV^{-1}$ or $\IQ^{-1}$ are designed to balance the cost of gradient evaluation, matrix-vector multiplication and preconditioning solver computation.

\section{Numerical Examples}\label{sec: numerical ex}

In order to verify our theoretical results, we test for numerical experiments arising in numerical partial differential equations, optimal control problems and data science. In Table \ref{tab: Summary of numerical examples.}, we summarize the tested problems and algorithms for reference.

\begin{table}[htp]
\footnotesize
	\centering
	\caption{Summary of numerical examples.} 
	\renewcommand{\arraystretch}{1.25}
	\begin{tabular}{@{} c c c  c  @{}}
	\toprule
  Problem   & Tested algorithms  & Recommended algorithm 
\smallskip \\  \hline 

 Quadratic problems    & \makecell{explicit Euler scheme \eqref{eq: EE gradient}, \\ GSS method \eqref{eq: nonlinear SOR 1} and AGSS method \eqref{eq:explicit}}& \makecell{ $\tilde {\kappa}(\mathcal N)\gg \kappa(F) $: \eqref{eq: nonlinear SOR 1},\\ $\kappa(F) \gg \tilde {\kappa}(\mathcal N)$:\eqref{eq:explicit}}

\smallskip \\

 Convection-diffusion model    & \makecell{IMEX schemes \eqref{eq:semi-implicit} and \eqref{eq:inexact-implicit} \\ of AGSS method}   &  inexact IMEX scheme \eqref{eq:inexact-implicit} 
\smallskip \\

Empirical risk minimization & GSS method \eqref{eq:GSS_saddle} and AGSS method \eqref{eq:explicit saddle}  &  \makecell{ $\tilde {\kappa}(\mathcal N)\gg \kappa(F) $:  \eqref{eq:GSS_saddle},\\ $\kappa(F) \gg \tilde {\kappa}(\mathcal N)$: \eqref{eq:explicit saddle}}

%
\smallskip \\
\bottomrule
	\end{tabular}
	\label{tab: Summary of numerical examples.}
\end{table}

\subsection{Quadratic Problems}\label{sec: Quadratic problems}
We consider synthetic
quadratic problems of the form
\begin{equation}\label{eq: quadratic problem}
    \min_x \frac{1}{2}(x, Lx) - (b, x)
\end{equation}
where $L$ is a positive definite matrix. Then \eqref{eq: quadratic problem} admits the unique solution such that $Lx^{\star} = b$. We decompose $L$ into symmetric part $A$ and skew-symmetric part $\mathcal N$:
$$
A = \frac{1}{2}(L+L^{\intercal}), \quad \mathcal N = \frac{1}{2}(L - L^{\intercal}).
$$
We test GSS method \eqref{eq: nonlinear SOR 1} and AGSS method \eqref{eq:explicit} and stop the iteration until the infinite norm of the error is less than $10^{-6}$.

We randomly generate the matrices $A$ and $\mathcal N$ with different condition numbers and compute the averaged iteration numbers (in $5$ samples). We list the dependence of the condition number vs iteration numbers. 
As one iteration of GSS and AGSS takes similar operations, the CPU time of each method is proportional to the iteration numbers. 

First we set $\kappa(A) = 2$ and vary $\tilde{\kappa} (\mathcal N) = 10, 20, 40, 80$. Then $\kappa(L) \approx \tilde{\kappa} (\mathcal N)$, which means the convergence rate is dominated by the condition number of the skew-symmetric part. In Fig. \ref{fig: EE and GSS rate}, we plot $\tilde{\kappa}(\mathcal N)$ versus averaged iteration number.
The increase of iteration number shows the convergence rate of the explicit Euler scheme is proportional to $\tilde{\kappa}^2 (\mathcal N) $ while that of GSS method is proportional to $\tilde {\kappa} (\mathcal N)$. 

\begin{figure}
     \centering
     \begin{subfigure}[b]{0.45\textwidth}
         \centering
    \includegraphics[scale = 0.3]{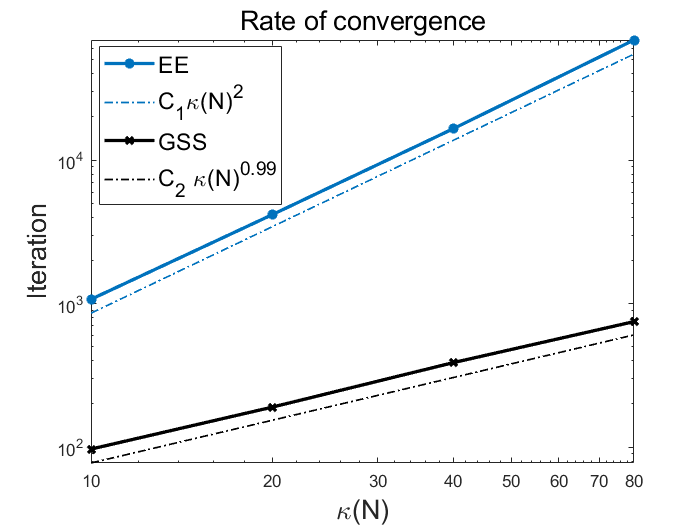}
\caption{$\tilde{\kappa}(\mathcal N)$ versus iteration number (averaged in 5 runs) for using the explicit Euler scheme \eqref{eq: EE gradient} and GSS method \eqref{eq: nonlinear SOR 1}.}
    \label{fig: EE and GSS rate}
     \end{subfigure}
     \hfill
     \begin{subfigure}[b]{0.45\textwidth}
         \centering
    \includegraphics[scale = 0.3]{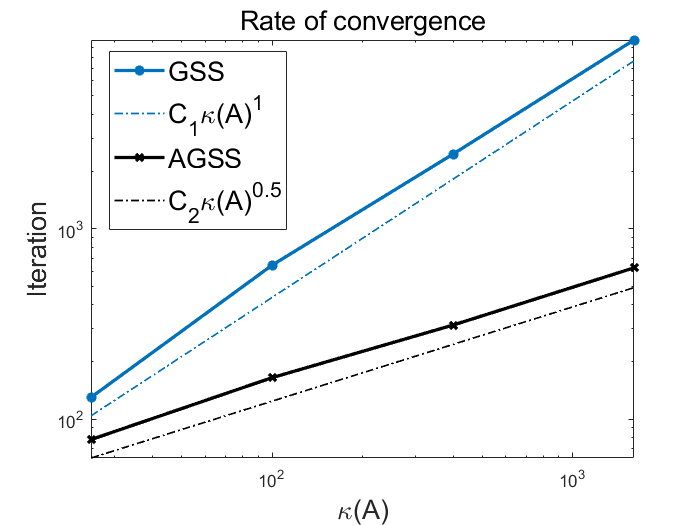}
\caption{$\kappa(A)$ versus iteration number (averaged in 5 runs) for using GSS method \eqref{eq: nonlinear SOR 1} and AGSS method \eqref{eq:explicit}.}
   \label{fig: GSS and AGSS rate}
     \end{subfigure}
     \hfill
        \caption{Convergence rate of GSS methods and AGSS methods.}
        \label{fig:three graphs}
\end{figure}

To show the acceleration we fix $\tilde \kappa(\mathcal N) = 2$ and vary $\kappa (A) = 25, 100, 400, 1600$. In this case, $\kappa (L) \approx \kappa(A)$. Therefore the condition number of the symmetric part dominates the convergence rate. We plot $\kappa(A)$ versus iteration number of using GSS method \eqref{eq: nonlinear SOR 1} and AGSS method \eqref{eq:explicit} in Fig. \ref{fig: GSS and AGSS rate}. The growth of AGSS iteration number is half of GSS iteration number, where the latter one is proportional to $\kappa (A)$. As the problem becomes ill-conditioned, AGSS is much more efficient.


\subsection{Convection-diffusion Model}\label{sec: Convection-diffusion model}
The second test is for the convection-diffusion equation
$$
-\Delta u  +\beta \cdot \nabla u = f,
$$
on the unit square $\Omega = [0,1] \times [0,1]$, with the Dirichlet boundary conditions and $\beta = (10, 10)^{\intercal}$. When applied the linear finite element discretization on uniform meshes with grid size $h$, we get the finite-dimensional linear system:
$$
L u = f.
$$
All experiments are implemented using the software package $i$FEM~\cite{chen2008ifem}.  
All tests are started from the zero vector, performed in MATLAB with double precision, and terminated when the current iterate satisfies the infinite norm of residual less than $\text{TOL} = 10^{-7}$. 

We compare the IMEX scheme \eqref{eq:semi-implicit} of AGSS method and the HSS method~\cite{bai2003hermitian}. One iteration of HSS method follows: given $u_k$, compute $u_{k+1}$ by
\begin{equation*}
    \begin{aligned}
        (\alpha I + A)u_{k+1/2} &=  (\alpha I - \mathcal N)u_{k} + f, \\
        (\alpha I + \mathcal N)u_{k+1} &=  (\alpha I - A)u_{k+1/2} + f
    \end{aligned}
\end{equation*}
where $A$ denotes the symmetric part from the diffusive terms, $\mathcal N$ is the skew-symmetric part corresponding to the convective terms, and the parameter $\alpha$ is set according to~\cite{bai2003hermitian} $\alpha = \sqrt{\lambda_{\min}(A)\lambda_{\max}(A)}$. 

For the AGSS scheme, we set the step size $\alpha = \sqrt{1/\kappa(A)}$. For both HSS and AGSS, we use direct solver in MATLAB to compute the matrix inverse. We see from Table \ref{table: HSS and AGSS iteration} that the iteration number of AGSS is more than that of HSS. We list the time complexity of inner solver in Table \ref{table: HSS and AGSS inner solver time} for $h = 1/128$ and $h = 1/256$. Though AGSS is free of solving $(\beta I + A)^{-1}$, herein computing the inverse of the skew-symmetric part is more time-consuming. 

To reduce the computational cost of the skew-symmetric part, we explore inexact AGSS (iAGSS) and inexact HSS (iHSS), utilizing the biconjugate gradient stabilized method (BiCGSTAB)~\cite{van1992bi} as the inner solver. Inexact inner solvers generally offer improved accessibility and efficiency.  In each step of iAGSS, the inner solver iterates until the tolerance is below $\text{TOL}=10^{-7}$ or for a maximum of 20 steps to compute $(\beta I + \mathcal N)^{-1}$. It is worth noting that iHSS requires the inner solver to achieve a tolerance of $\text{TOL}=10^{-9}$; otherwise, the outer iteration fails to converge. 

\begin{table}[htp]
	\centering
    \caption{Iteration steps (CPU time in seconds) for the convection-diffusion model problem. }
    \label{table: HSS and AGSS iteration}
	\renewcommand{\arraystretch}{1.125}
	\begin{tabular}{@{} c c c  c c c c @{}}
	\toprule
$h$	&	DoF	&	HSS	&	AGSS	&	iHSS	&	iAGSS
\smallskip \\  \hline 

1/32	&	961	&	133 (0.43s)	&	295 (0.71s)	&	133 (0.22s)	&	307 (0.22s)

\smallskip \\

1/64	&	3969	&	269 (4.6s)&	550	(7.6s)&	269 (1.5s)	&	550 (0.81s)

\smallskip \\

1/128	&	16129	&	536	(48s)&	1036 (64s)&	536 (17s)	&	1036 (7.1s)

\smallskip \\
 
1/256	&	65025	&	1078 (431s)	&	1949 (569s)	&	1078 (140s)	&	1948 (47s)	
\smallskip \\
\bottomrule
	\end{tabular}
\end{table}


\begin{table}[htp]
	\centering
    \caption{Inner solver time (in seconds) for the convection-diffusion model problem. }
    \label{table: HSS and AGSS inner solver time}
	\renewcommand{\arraystretch}{1.125}
	\begin{tabular}{@{}c || c c c  c || c c c c @{}}
	\toprule
		&  \multicolumn{4}{c||}{$h = 1/128$}&  \multicolumn{4}{c}{$h = 1/256$}\\
		\cline{2-9}	
	&	HSS	&	AGSS	&	iHSS	&	iAGSS & 	HSS	&	AGSS	&	iHSS	&	iAGSS \\  \hline 

$(\beta I + \mathcal N)^{-1}$	&	34	&	63	&	4.1	&	6.2	&	316	&	561	&	25	&	38	 
 \\

$(\beta I + A)^{-1}$	&	12	&	-	&	12	&	- &	110	&	-	&	109	&	-
 \\

Total	&	48	&	64	&	17	&	7.1	&	431	&	569	&	140	&	47	
 \\
\bottomrule
	\end{tabular}
\end{table}

Table \ref{table: HSS and AGSS iteration} reveals that iAGSS has a comparable iteration count to AGSS, yet iAGSS significantly outperforms other methods in terms of CPU time. The increase in iteration steps and CPU time underscores that AGSS and HSS share the same convergence rate. Examining the time complexity of the inner solver in Table \ref{table: HSS and AGSS inner solver time}, we observe that inexact solvers entail substantially lower computational costs. iAGSS emerges as the most efficient method, demonstrating superior performance in both solver times for symmetric and skew-symmetric parts.



\subsection{Empirical Risk Minimization} \label{sec: Empirical risk minimization}
Consider the regularized empirical risk minimization (ERM) with linear predictors, which is a classical supervised learning problem. Given a data matrix $B = [b_1, b_2, \cdots, b_n]^{\intercal} \in \mathbb{R}^{n \times m}$ where $b_i \in \mathbb{R}^m$ 
is the feature vector of the $i$-th data entry, the ERM problem aims to solve
$$
\min _{u \in \mathbb{R}^m} g(Bu)+f(u),
$$
where $g: \mathbb{R}^n \to \mathbb{R}$ is some convex loss function, $f: \mathbb{R}^m \to \mathbb{R}$ is a convex regularizer and $u \in \mathbb{R}^m$ is the linear predictor. Equivalently, we can solve the dual problem $\max _{p \in \mathbb{R}^n}\left\{-g^*(p)-f^*\left(-B^{\intercal} p\right)\right\}$ or the saddle point problem $$\min _{u \in \mathbb{R}^m} \max _{p \in \mathbb{R}^n}\left\{p^{\intercal} Bu-g^*(p)+f(u)\right\}.$$ The saddle point formulation is favorable in many scenarios, e.g., when such formulation admits a finite-sum structure~\cite{wang2017exploiting}, reduces communication complexity in the distributed setting~\cite{xiao2019dscovr} or exploits sparsity structure~\cite{lei2017doubly}.

We define the loss function as the mean square error, given by $g^*(p) = \frac{1}{2}\|p\|_G^2 $, induced by the $G$-norm associated with an SPD matrix $G$. Additionally, we incorporate an $\ell_2$ regularizer, where $f(u) = \frac{1}{2}\|u\|^2$. The ERM model is then formulated as a linear saddle point problem. In this formulation, we have $\nabla^2 F = \begin{pmatrix}
I & ~ 0 \\
0 & ~ G 
\end{pmatrix} $, and $\mathcal B^{\sym} = \begin{pmatrix}
0 & ~ B^{\intercal} \\
B & ~ 0
\end{pmatrix}$.

For our numerical experiments, we consider two cases:
\begin{itemize}
    \item Case 1: $\kappa(\mathcal B^{\sym}) \gg \kappa(G)$; 

    \item Case 2: $\kappa(\mathcal B^{\sym}) \ll \kappa (G)$.

    \item Case 3: $\kappa(\mathcal B^{\sym}) = \sqrt{\kappa (G)} $.
\end{itemize}
We implement the GSS method \eqref{eq:GSS_saddle} and AGSS method \eqref{eq:explicit saddle} with $\IV$ and $\IQ$  set as identities. Additionally, we provide a list of other first-order primal-dual algorithms for comparison:
\begin{itemize}
    \item Lifted primal-dual (LPD) method  for bilinearly coupled smooth minimax optimization proposed in~\cite{thekumparampil2022lifted}.

    \item AcceleratedGradient-OptimisticGradient (AG-OG) method~\cite{li2023nesterov} with restarting regime.

    \item Accelerated primal-dual gradient (APDG) method~\cite{kovalev2022accelerated} for smooth and convex-concave saddle-point problems with bilinear coupling.
    
    \item Optimistic gradient descent ascent (OGDA)  method~\cite{mokhtari2020convergence} in smooth convex-concave saddle point problem.
\end{itemize}

Among all the tests, we set the parameters as specified. AGSS, LPD, AG-OG, and APDG are first-order algorithms achieving optimal convergence rates, while OGDA obtains a sublinear rate for general saddle point problems. 

In Table \ref{tab: opt test 1}, we can observe that the performance of GSS and AGSS is comparable as the non-symmetric part dominate, i.e. $\kappa(\mathcal B^{\sym}) \gg \kappa(G)$, the acceleration is not helpful. However, when $\kappa(G)$ are relatively large, as shown in Table \ref{tab: opt test 2} and Table \ref{tab: opt test 3}, AGSS is much faster than GSS due to the acceleration technique. 

Overall, AGSS is more stable and efficient concerning iteration number and CPU time, especially when the condition number $\kappa(G)$ is large.

\begin{table}[htp]
\small
	\centering
  \caption{Iteration number (CPU time in seconds) for empirical risk minimization problem. $m = 2500, n =500$ and $\kappa(G) =2$.}
    \label{tab: opt test 1}
	\renewcommand{\arraystretch}{1.25}
	\begin{tabular}{@{} c c c  c c c c @{}}
	\toprule
$\kappa(\mathcal B^{\sym})$	&	GSS	&	AGSS	&	LPD	&	AG-OG	&	APDG	&	OGDA	
\smallskip \\  \hline 

$ 100$	&	1403 (0.63s)	&	1704 (0.73s)	&	2825 (1.2s)	&	3643 (2.1s)	&	5627 (4.4s)	&	2825 (1.6s)
\smallskip \\

$200$	&	2811 (1.2s)	&	3399 (1.5s)	&	5681 (2.4s)	&	5931 (3.6s)	&	11260 (8.7s)	&	5641 (3.4s)

\smallskip \\

$400$	&	5627 (2.1s)	&	6789 (2.7s)	&	11273 (4.3s)	&	9400 (5.3s)	&	22525 (17s)	&	11273 (6.3s)

\smallskip \\
$800$	&	11260 (4.3s)	&	13569 (5.6s)	&	22539 (9.3s)	&	19050 (11s)	&	45055 (34s)	&	22539 (12s)	
\smallskip \\
\bottomrule
	\end{tabular}
\end{table}


\begin{table}[htp]
\small
	\centering
    \caption{Iteration number (CPU time in seconds) for empirical risk minimization problem. $m = 2500, n =500$ and $\kappa(\mathcal B^{\sym}) = 2$. `$-$' means the convergence failed in 50000 steps.}
    \label{tab: opt test 2}
	\renewcommand{\arraystretch}{1.25}
	\begin{tabular}{@{} c c c  c c c c @{}}
	\toprule
$\kappa(G)$	&	GSS	&	AGSS	&	LPD	&	AG-OG	&	APDG	&	OGDA	
\smallskip \\  \hline 

$400$	&	4412 (1.7s)	&	192 (0.077s)	&	153 (0.071s)	&	827 (0.51s)	&	2750 (2.1s)	&	8843 (4.7s)
\smallskip \\

$800$	&	8838 (3.6s)	&	268 (0.11s)	&	177 (0.077s)	&	1304 (0.80s)	&	5493 (4.8s)	&	17695 (9.5s)
\smallskip \\

$1600$	&	17689 (7.0s)	&	377 (0.16s)	&	201 (0.08s)	&	2035 (1.2s)	&	10980 (8.4s)	&	35397 (20s)

\smallskip \\
$ 3200$	&	35366 (14s)	&	531 (0.24s)	&	274 (0.12s)	&	3399 (2.1s)	&	21953 (17s)	&	$>$50000		
\smallskip \\
\bottomrule
	\end{tabular}
\end{table}


\begin{table}[htp]
\small
	\centering
    \caption{Iteration number (CPU time in seconds) for empirical risk minimization problem. $m = 2500, n =500$ and $\kappa(\mathcal B^{\sym}) = \sqrt{\kappa(G)}$. `$-$' means the convergence failed in 50000 steps.}
    \label{tab: opt test 3}
	\renewcommand{\arraystretch}{1.25}
	\begin{tabular}{@{} c c c c  c c c c @{}}
	\toprule
$\kappa(G)$	&	$\kappa(B)$	&	GSS	&	AGSS	&	LPD	&	AG-OG	&	APDG	&	OGDA
\smallskip \\  \hline 
$160^2$	&	160	&	-	&	1815 (1.0s)	&	2162 (1.5s)	&	4920 (3.2s)	&	4436 (4.2s)	&	-	\smallskip \\
$320^2$	&	320	&	-	&	3346 (1.5s)	&	3911 (2.2s)	&	9804 (7.7s)	&	8061 (6.4s)	&	-	\smallskip \\
$640^2$	&	640	&	-	&	6130 (2.4s)	&	7007 (2.9s)	&	19501 (11s)	&	14532 (11s)	&	-	\smallskip \\
$1280^2$	&	1280	  &	-	&	11143 (5.9s)	&	12394 (6.4s)	&	38618 (25s)	&	25939 (23s)	&	-	
\smallskip \\
\bottomrule
	\end{tabular}
\end{table}

\section{Conclusion}\label{sec: conclusion}
In this paper, we propose GSS methods and AGSS methods for solving a class of strongly monotone operator equation $\mathcal A(x) = 0$ based on the splitting $\mathcal A(x) = \nabla F(x) + \mathcal Nx$. The proof of strong Lyapunov property bridges the design of dynamical system on the continuous level, the choice of Lyapunov function and the linear convergence of the numerical schemes (as discretization of the continuous flow). As direct applications, we derive optimal algorithms for strongly-convex-strongly-concave saddle point systems with bilinear coupling.

As it is well-known in the convex minimization, the optimal mixed type convergence rates is $\mathcal O(\min\{1/k^2, (1-1/\sqrt{\kappa(f)})^{-k})\}$~\cite{chen2021unified}, which enjoys a sub-linear rate as long as $f$ is convex only, i.e., $\mu = 0$. With more careful design of parameters, our AGSS may also achieve accelerated sub-linear convergence rates. Such method can be used as a smoother for deriving nonlinear multigrid methods, yielding iteration complexity free of problem size and condition number.


\bibliographystyle{amsplain}
\bibliography{Optimization}  





\end{document}